\documentclass{article}

\usepackage{microtype}
\usepackage{graphicx}
\usepackage{subfigure}
\usepackage{booktabs} 

\usepackage{hyperref}
\usepackage{bm} 
\usepackage{bbm}
\usepackage{amsmath}
\usepackage{amsthm}
\usepackage{amsfonts}
\usepackage{amssymb}
\usepackage{pbox}
\usepackage{comment}
\usepackage{verbatim} 
\usepackage{enumitem}
\usepackage[Symbolsmallscale]{upgreek}

\usepackage[accepted]{icml2023}

\usepackage{amsmath}
\usepackage{amssymb}
\usepackage{mathtools}
\usepackage{amsthm}

\usepackage[capitalize,noabbrev]{cleveref}

\newcommand{\E}{\mathbb{E}}
\newcommand{\R}{\mathbb{R}}
\newcommand{\N}{\mathbb{N}}

\newcommand{\Char}{\mathbbm{1}}
\newcommand{\ones}{\mathbf{1}}
\newcommand{\mM}{\mathbf{M}}

\newcommand{\mW}{\mathbf{W}}
\newcommand{\mD}{\mathbf{D}}

\newcommand{\mJ}{\mathbf{J}}
\newcommand{\mQ}{\mathbf{Q}}
\newcommand{\mP}{\mathbf{P}}
\newcommand{\mV}{\mathbf{V}}
\newcommand{\mU}{\mathbf{U}}
\newcommand{\mF}{\mathbf{F}}
\newcommand{\mS}{\mathbf{S}}
\newcommand{\mA}{\mathbf{A}}
\newcommand{\mB}{\mathbf{B}}

\newcommand{\mK}{\mathbf{K}}
\newcommand{\mZ}{\mathbf{Z}}
\newcommand{\eye}{\mathbf{I}}

\newcommand{\0}{\mathbf{0}}
\newcommand{\vd}{\mathbf{d}}
\newcommand{\vx}{\mathbf{x}}
\newcommand{\vy}{\mathbf{y}}

\newcommand{\ve}{\mathbf{e}}

\newcommand{\vv}{\mathbf{v}}

\newcommand{\vu}{\mathbf{u}}
\newcommand{\vf}{\mathbf{f}}
\newcommand{\vg}{\mathbf{g}}
\newcommand{\vh}{\mathbf{h}}

\newcommand{\prob}{\mathbb{P}}

\newcommand{\cG}{\mathcal{G}}
\newcommand{\cN}{\mathcal{N}}

\newcommand{\cE}{\mathcal{E}}

\newcommand{\cL}{\mathcal{L}}
\newcommand{\cO}{\mathcal{O}}

\newcommand{\cK}{\mathcal{K}}
\newcommand{\Int}{\text{Int}}

\newcommand{\vpi}{\boldsymbol{\pi}}
\newcommand{\vmu}{\boldsymbol{\mu}}
\newcommand{\vnu}{\boldsymbol{\nu}}

\newcommand{\bfdelta}{\boldsymbol \delta}
\newcommand{\bfgamma}{\boldsymbol \gamma}

\newcommand{\bfepsilon}{\boldsymbol \epsilon}

\theoremstyle{plain}
\newtheorem{theorem}{Theorem}[section]
\newtheorem{proposition}[theorem]{Proposition}
\newtheorem{lemma}[theorem]{Lemma}
\newtheorem{corollary}[theorem]{Corollary}
\theoremstyle{definition}

\newtheorem{assumption}{}

\theoremstyle{remark}
\newtheorem{remark}[theorem]{Remark}

\usepackage[textsize=tiny]{todonotes}

\icmltitlerunning{SRRW on General Graphs}

\begin{document}

\twocolumn[
\icmltitle{Self-Repellent Random Walks on General Graphs - \\
                Achieving Minimal Sampling Variance via Nonlinear Markov Chains}



\icmlsetsymbol{equal}{*}

\begin{icmlauthorlist}
\icmlauthor{Vishwaraj Doshi}{to}
\icmlauthor{Jie Hu}{goo}
\icmlauthor{Do Young Eun}{goo}
\end{icmlauthorlist}

\icmlaffiliation{to}{IQVIA Inc., Plymouth Meeting, USA.}
\icmlaffiliation{goo}{Department of Electrical and Computer Engineering, North Carolina State University, Raleigh, USA.}

\icmlcorrespondingauthor{Do Young Eun}{dyeun@ncsu.edu}

\icmlkeywords{Nonlinear Markov Chains, Markov Chain Monte Carlo, General Graph, Stochastic approximation,  Graph Sampling}

\vskip 0.3in
]



\printAffiliationsAndNotice{}  

\begin{abstract}
We consider random walks on discrete state spaces, such as general undirected graphs, where the random walkers are designed to approximate a target quantity over the network topology via sampling and neighborhood exploration in the form of Markov chain Monte Carlo (MCMC) procedures. 
Given any Markov chain corresponding to a target probability distribution, we design a \textit{self-repellent random walk} (SRRW) which is less likely to transition to nodes that were highly visited in the past, and more likely to transition to seldom visited nodes. 
For a class of SRRWs parameterized by a positive real $\alpha$, we prove that the empirical distribution of the process converges almost surely to the the target (stationary) distribution of the underlying Markov chain kernel. 
We then provide a central limit theorem and derive the exact form of the arising asymptotic co-variance matrix, which allows us to show that the SRRW with a stronger repellence (larger $\alpha$) always achieves a smaller asymptotic covariance, in the sense of Loewner ordering of co-variance matrices. 
Especially for SRRW-driven MCMC algorithms, we show that the decrease in the asymptotic sampling variance is of the order $O(1/\alpha)$, eventually going down to zero. 
Finally, we provide numerical simulations complimentary to our theoretical results, also empirically testing a version of SRRW with $\alpha$ increasing in time to combine the benefits of smaller asymptotic variance due to large $\alpha$, with empirically observed faster mixing properties of SRRW with smaller $\alpha$.
\end{abstract}
\section{Introduction} \label{section:introduction}

Random walk based techniques are a staple in statistics and learning theory. Markov chains such as the Metropolis Hastings random walk, designed to achieve any given target probability distribution as its stationary measure, are widely used as Markov chain Monte Carlo (MCMC) samplers and in distributed optimization via stochastic gradient descent \citep{sun2018on, hu2022efficiency_arxiv}. The local nature of the information required to compute state transition probabilities means that the algorithms scale well and are robustly implementable over state spaces such as large graphs/networks with general topologies. However, classic Markov chains can often be victims of limitations set by the underlying topology of the state space (communication matrix or adjacency matrix of the underlying network structure) leading to correlated samples which can negatively affect the estimator performance. It has also been well established that the time-reversibility requirement for the classical MCMC samplers is one of the causes for their slow convergence \citep[see][Section 1]{andrieu2021peskun}. One way in which this problem has been approached in the literature is via construction of non-reversible versions of the base Markov chain \citep{diaconis2000analysis, turitsyn2011irreversible, chen2013accelerating, ma2016unifying, thin2020nonreversible}, which is often done by inducing some form of \textit{non-backtracking} behaviour, that is, avoiding states most recently visited by the random walker \citep{ alon2007non}. This involves the random walker interacting with some of its own past history, and has been shown to possess better efficiency than the original \textit{base} Markov chain in the sense of the MCMC estimator achieving a smaller asymptotic variance \citep{neal2004improving, LeeSIGMETRICS12}. Since these non-backtracking based methods only utilize the most \textit{recent} history of the random walker and are still provably more efficient, it is natural to consider the design of protocols where the random walker interacts with its \textit{entire} past history to speed up its diffusion and increase its sampling efficiency, especially for sampling over discrete state spaces. This is the approach taken in our paper.

Let $\cG(\cN,\cE)$ be an undirected, connected graph where $\cN \triangleq \{1,\cdots, N\}$ denotes the set of nodes and $\cE$ denotes the set of edges, where we say $(i,j)\in\cE$ if there is an edge between nodes $i,j\in\cN$. We use $\mA=[a_{ij}]_{i,j\in\cN}$ to represent the adjacency matrix of the graph, where $a_{ij}>0$ if $(i,j)\in\cE$, and zero otherwise; $\cN(i) \triangleq \{ j \in \cN ~|~ (i,j)\in\cE \}$ refers to the set of neighbors of node $i$; $\text{deg}(i) \triangleq \sum_{j \in \cN} a_{ij}$ will refer to the degree of each node $i\in\cN$. Denote by $\Sigma$ the $N$-dimensional probability simplex over $\cN$, with $\Int(\Sigma)$ denoting its interior, and let $\mP \triangleq [P_{ij}]_{i,j\in\cN}$ be the transition probability matrix of an ergodic, time-reversible Markov chain over $\cN$, with its stationary distribution $\vmu \triangleq [\mu_i]_{i \in \cN}$. Without loss of generality, we assume $P_{ij} > 0$ if and only if $a_{ij}>0$. In this setup, we design \textit{Self-Repellent Random Walks} (SRRWs) on general graphs\footnote{We consider graphs because they represent a generalization of (discrete) finite state spaces by imposing a communication (adjacency) matrix. The existence of an edge between two nodes (states) represents a non-zero probability of state transitions between the two nodes.} indexed by a tunable parameter $\alpha \geq 0$, all of which can sample from $\vmu\in\Sigma$, and then study their sampling `efficiency' as a function of $\alpha$ (with $\alpha = 0$ being equivalent to the baseline Markov chain with transition kernel $\mP$).

\paragraph{The SRRW transition kernel:}
Consider the Markov chain kernel (transition matrix) $\mK[\vx] \triangleq [K_{ij}[\vx]]_{i,j\in\cN}$, whose transition probabilities are mappings $K_{ij}:\Sigma \to [0,1]$, given by
\begin{equation}\label{eqn:general srrw kernel}
    K_{ij}[\vx] \triangleq \frac{P_{ij} r_{\mu_j}(x_j)}{\sum_{k\in\cN}P_{ik} r_{\mu_k}(x_k)},
\end{equation}
for any probability vector $\vx \triangleq [x_i]_{i \in \cN} \in \Sigma$. Here, $\{ r_{\mu_i} \}_{i \in \cN}$ is a family of positive functions $r_{\mu_i}:[0,1]\to\R_+$ parameterized by $\mu_i$, with $r_{\mu_i}(x_i)$ \textit{decreasing} in $x_i \in [0,1]$ and $r_{\mu_i}(\mu_i) = C$, for all $i \in \cN$.\footnote{As we shall see later, $x_i$ will be directly proportional to the visit count to any node $\in\cN$, since $\vx\in \Sigma$ will be the empirical distribution of the self-repellent random walk.} Transition probability kernels defined in this fashion, taking probability distributions as argument, are called \textit{`nonlinear'} Markov kernels \citep{Andrieu2007Nonlinear,Andrieu2011OnNonlinearMCMC}, as opposed to classical Markov chains with kernels $\mP$ that are often interpreted as linear operators -- the transition probabilities at each step being independent of $\vx$ (i.e., the case where $r_{\mu_i}(\cdot)$ is a constant function).

Stochastic processes utilizing nonlinear Markov kernels are called nonlinear Markov chains, and can be simulated/generated using \textit{self-interacting Markov chains} (SIMCs) \citep[see][]{Moral2004OnConvergence, Moral2006SelfInteracting, Moral2010Interacting}. Let $\{X_n\}_{n \geq 0}$ be a random walker over $\cN$, and let $\vx_n$ be its \textit{occupational measure} or \textit{historical empirical distribution} up to time $n\geq 0$, written as 
\begin{equation}\label{eqn:empirical_dist_1}
    \vx_n \triangleq \frac{1}{n+1}\sum_{k = 0}^n \bfdelta_{X_k},
\end{equation}
where $\bfdelta_{X_k}$ is the delta measure whose $X_k$'th entry is one and the rest are zero, thus recording the position of the random walker at time $k\geq 0$. The process $\{X_n\}_{n \geq 0}$ becomes a SIMC if at each time step $n\geq 0$, the random walker makes transitions according to some nonlinear kernel $\mK[\vx_n]$, not necessarily as defined in \eqref{eqn:general srrw kernel}. We say that the process $\{X_n\}_{n \geq 0}$ is a SRRW if it is a SIMC with $\mK[\vx]$ as in \eqref{eqn:general srrw kernel}. We use the term \textit{self-repellent} since at each time step, the transition probability to a node $j\in\cN$ is proportional to $r_{\mu_j}([\vx_n]_j)$ where $[\vx_n]_j = \frac{1}{n}\sum_{k=1}^{n}  \Char_{\{X_k=j\}}$, and is thus a decreasing function of the visit count to $j\in\cN$. In other words, the walker is less likely to move to a node that has been visited more often so far (thus self-repellent).

When $P_{ij} \propto a_{ij}$ in \eqref{eqn:general srrw kernel} for each $i\in\cN$, the SRRW is a self-repellent version of the well-known \textit{simple random walk} (SRW) procedure, with the target distribution being proportional to the degree of the nodes, that is, $\mu_i \propto \deg(i)$ for all $i\in\cN$. Like most general MCMC procedures, the SRRW kernel can also be defined for \textit{any given} sampling distribution $\vmu\in\Int(\Sigma)$, for instance, by setting $\mP$ to be the transition matrix of a Metropolis Hastings Random Walk (MHRW) with stationary distribution $\vmu$. For example if $\mu_i = 1/N$, that is $\vmu = \frac{1}{N}\ones$ -- the uniform distribution over the set of nodes $\cN$, then we can choose $P_{ij} = \min \left\{ \frac{1}{\deg(i)},\frac{1}{\deg(j)} \right\}$ for all $(i,j)\in\cE$, with $P_{ii} = 1 - \sum_{j\neq i} P_{ij}$. The matrix $\mP$ defined in this manner is the MHRW kernel with the uniform distribution as its stationary measure, and is among the most commonly used kernels for unbiased graph sampling \cite{li2015random} and distributed optimization \cite{sun2018on}. The elegance in the Metropolis Hastings algorithm and the key to its widespread adaptation lies in the fact that at each time step, the entries of $\vmu$ need only to be known for the neighbouring nodes of the random walkers (that is, only \textit{local} information required), and only up to a constant multiple. This property ensures a robust, scalable implementation of the MHRW, since global constants are often unknown for large networks a priori.

Our SRRW construction begins with $r_{\mu_i}(\cdot)$ taking a polynomial form for all $i \in \cN$, given by
\begin{equation} \label{eqn:polynomial_repellence}
    r_{\mu_i}(x_i) \triangleq \left(\frac{x_i}{\mu_i}\right)^{-\alpha}, ~~~~~\forall ~\alpha \geq 0,
\end{equation}
where the parameter $\alpha\geq 0$ can be perceived as the \textit{strength of the self-repellence mechanism} designed into the SRRW transition kernel. Similar to the MHRW transition kernel, only the local information regarding entries of $\vmu$ needs to be known at any given time step, and up to only a constant multiple. For convenience, we formalize this property as \textit{scale-invariance (S.I.)}: an SRRW kernel posesses S.I. if for all $i,j\in\cN$
\begin{enumerate}
    \item[(i)] Computing $K_{ij}[\vx]$ only requires knowing $\mu_k$ for $k \in \cN(i)$, and only up to a constant multiple for any $i\in\cN$;
    \item[(ii)] $K_{ij}[C'\vx] = K_{ij}[\vx]$ for any constant $C' > 0$. 
\end{enumerate}
Indeed, we show in Appendix \ref{appendix:discussion on polynomial} that out of all possible forms for the functions $r_{\mu_i}(\cdot)$, only the polynomial form as in \eqref{eqn:polynomial_repellence} possesses the S.I. property. Henceforth, we restrict ourselves to the polynomial form of $r_{\mu_i}(\cdot)$.

\paragraph{Our Contribution:}

\begin{enumerate}
    \item We show that given any MCMC kernel $\mP$ which samples from a target distribution $\vmu$, the corresponding SRRW is asymptotically more \textit{efficient} as a random walk based sampler. We do this by first showing that 
        \begin{equation}
            \vx_n \xrightarrow[n \to \infty]{\text{a.s.}} \vmu, ~~~~~\forall ~\alpha \geq 0.
        \end{equation}
    We then provide second-order convergence results in the form of a central limit theorem (CLT); that is, we show there exists an \textit{asymptotic co-variance matrix} $\mV(\alpha) \in \R^{N \times N}$ parameterized by $\alpha \geq 0$, such that
    \begin{equation}
        \sqrt{n}(\vx_n - \vmu) ~\xrightarrow[dist.]{n\to \infty}~ N(\0,\mV(\alpha)).
    \end{equation}
    We obtain these results by first viewing the SRRW as a stochastic approximation (SA) algorithm with state-dependent noise \citep{kushner1997,fort2015central}, allowing us to form a connection between the stochastic process and a deterministic system of ordinary differential equations (ODEs). We establish global convergence results for this ODE system, which lay the foundations for proving almost sure convergence of the stochastic SRRW process.

    \item For any $\alpha \geq 0$ we derive the \textit{exact} form of $\mV(\alpha)$ in terms of $\alpha$ and the spectrum (eigenvalues and eigenvectors) of $\mP$. This allows us to show that kernels parameterized by larger $\alpha$ are asymptotically more \textit{efficient} samplers; that is, they achieve smaller sampling variance. This is done by showing that the asymptotic covariance matrices follow a \textit{Loewner ordering}:\footnote{Matrices $\mA$, $\mB$ follow the Loewner ordering $\mA <_L \mB$ if $\mA \neq \mB$ and $\mB-\mA$ is positive semi-definite.}
    \begin{equation}\label{eqn:covariance_ordering_intro}
        \mV(\alpha_1) <_L \mV(\alpha_2), ~~~~~\forall ~ \alpha_1 > \alpha_2 \geq 0.
    \end{equation}
     In other words, as long as the numerical/computational stability of the random walk implementation can be ensured, larger values of $\alpha$ are always more favourable in terms of achieving a smaller (asymptotic) sampling variance. We also derive an upper bound on the ratio of its sampling variance over that of the baseline Markov chain, and show that this upper bound \textit{goes down to zero} as $\alpha \to \infty$ with speed $O(1/\alpha)$. This is surprising because asymptotically for large enough $\alpha$, the SRRW, which is a stochastic process whose trajectories are constrained by `walking' on the underlying communication matrix of the network, achieves smaller sampling variance than an \textit{i.i.d.} sampler\footnote{This corresponds to a sampler that can visit any node $i$ with probability $\mu_i$ independent of its previous position at any given time. Clearly, in the graph setting, this requires the sampler to 'jump' to any other node by ignoring the underlying network structure altogether - something  which random walkers on general graphs are not permitted to do.} whose variance is always a constant positive value.

    \item  We confirm our theoretical results by numerically observing the predicted asymptotic performance ordering of SRRW over a wide range of $\alpha \geq 0$ for the task of MCMC sampling over various graph topologies. To effectively handle potentially slower mixing of the SRRW process in 
    the initial transient period for large $\alpha$, we provide simulation results for an SRRW with \textit{time-varying} $\alpha$. Starting with smaller values of $\alpha$ which monotonically increase with time, we show that the empirically observed superior mixing of processes with smaller $\alpha$ can be combined with the theoretically proven asymptotic efficiency of SRRW with larger $\alpha$, resulting in a far more efficient MCMC algorithm. 

\end{enumerate}

\paragraph{Related Works:}
While numerous version of self-repelling walks on countable state spaces have been studied in the past, we are primarily interested in random walk kernels which do not absolutely forbid (with probability one) transitions to previously visited nodes. The SRRW defined in our paper, which still allows transitions to past visited nodes with positive probability, falls into a class of `weakly' self-avoiding random walks \citep{Amit1983asymptotic, Toth1995bond, veto2008self, grassberger2017self}. The works \textit{ibid.} study weakly self-avoiding random walks on topologies such as $1$ or $d$-dimensional lattices, and provide theoretical results on properties such as recurrence properties, escape times from sets or average cover times. Thus, the analytical focus is not on convergence properties of some statistical attributes of the random walks to a particular target, making it difficult to utilize these existing results to design algorithms for learning and statistical inference on \textit{general graphs}.

Works which do study convergence properties of empirical distributions of random walks with repellent dynamics consider processes where the repellence is between two \citep{Chen2014Two} or more \citep{rosales2022vertex} particles, and the results are again limited to random walks on complete graphs. Even thought the random walks therein are not self-repellent in our sense, the analysis techniques used are similar to ours, and are substantially influenced by the stochastic approximation framework in \cite{Benaim97vertex} for the analysis of vertex \textit{reinforced} random walks (VRRW) \citep[also see][]{benaim2011dynamics, benaim2012strongly}. As the name might suggest, the nonlinear Markov kernels considered in this literature are self-attractive, that is, the random walker transitions to more frequently visited nodes with greater probability. The formulation is very general, and the attraction is captured by generic, monotonically \textit{increasing} functions, which is in contrast to the monotonically \textit{decreasing} $r_{\mu_i}(x_i)$ terms in our SRRW kernel as in \eqref{eqn:general srrw kernel}. Even with such a general formulation, results showing convergence of the empirical measure to identifiable probability distributions are restricted to special cases such as complete graphs \citep{benaim2012strongly}. Consequently, the analytical focus is shifted towards providing \textit{localization} results, analyzing the property of the VRRW to get trapped inside subsets of the graph under consideration \citep{tarres2004vertex, angel2014localization}. Thus, even though the dynamics of VRRWs may look similar at first to the ones studied in our paper, the two types of random walks are essentially opposites of each other. The SRRW is specifically constructed for \textit{sampling from any arbitrarily given target} distribution over any \textit{general graph}, while the VRRW has found applications in fields such as ecology for understanding the behaviour of organisms like certain bacteria \citep{stevens1997aggregation, pemantle2007survey}, and can be interpreted as \textit{learning an unknown target} with potential applications to solving certain optimization problems \citep{avrachenkov2021dynamic} \textit{over a special graph} such as a grid or a complete graph. Note that the SRRW kernel with $\alpha<0$ in \eqref{eqn:polynomial_repellence} is actually self-attractive, and therefore a VRRW. Our analytical approach covers such cases, and our first and second order convergence results actually hold for all $\alpha$ greater than a certain negative threshold value. For values of $\alpha$ smaller than this threshold (that is, the dynamics are more self-attractive in nature), we numerically observe non-convergence to the target distribution. This is analogous to the phase-transition behaviour often observed for VRRWs \citep{pemantle1988phase,volkov2006phase,akian2007multiple} and puts our work in line with the broader literature of reinforced random walks.

As mentioned earlier, the SRRW described in this paper is a self-interacting Markov chain (SIMC), which can themselves be categorized under \textit{adaptive} MCMC algorithms, with the kernel parameter being the empirical distribution of the process itself \citep{fort2011convergence, fort2014central}. While numerous works study such processes on continuous state spaces \citep{Andrieu2007Nonlinear, andrieu2007convergence, andrieu2008note, Andrieu2011OnNonlinearMCMC}, they focus on the construction of kernels employing \textit{random jumps} to other (potentially distant) states depending on their historical visit counts, and are typically designed to tackle specific challenges such as sampling from multi-modal distributions and sequential Monte Carlo sampling, with performance improvements typically shown via simulation results. In practice, however, information regarding the shape of the target distribution such as its modality is unknown. These methods are also largely inapplicable for finite state spaces such as graphs where the underlying communication matrix determines the possible state transitions, therefore prohibiting jumps between non-neighbouring nodes.

\paragraph{Structure of the paper:}

In the rest of the paper, we first present some preliminaries and set up the SRRW procedure as a stochastic approximation in Section \ref{Section:algorithmic setup}. In Section 3, we then analyze a mean-field ODE system whose asymptotic dynamics are closely related to our stochastic process; proving uniqueness of the target distribution $\vmu \in \Int(\Sigma)$ as its fixed point and showing global convergence of its solutions to $\vmu$ with the help of a Lyapunov function.  Our main results corresponding to our contributions as discussed earlier in the introduction are provided in Section \ref{section:stochastic_analysis}, where we also discuss the application of our SRRW scheme to MCMC sampling. In Section \ref{section:numerical_results} we provide numerical results supporting our theoretical findings along with additional tests that showcase the advantages of employing the SRRW with time-varying $\alpha$, before concluding in Section \ref{section:conclusion}.
\section{Algorithmic setup}\label{Section:algorithmic setup}

In this section, we first standardize some basic notations which will be used throughout the paper, and review preliminary concepts regarding time-reversible Markov chains and their spectrum. We then introduce the SRRW iteration as a stochastic approximation algorithm along with a deterministic system whose asymptotic dynamics are closely related to our stochastic iteration.

\subsection{Preliminaries}\label{subsection: basic notations}

\paragraph{Basic Notations:}
We use lower-case boldface letters to denote column vectors (e.g. $\vv \triangleq [v_i]_{i\in\cN} \in \R^N$), and upper-case boldface letters to denote matrices (e.g. $\mM \triangleq [M_{ij}]_{i,j\in\cN} \in\R^{N\times N}$). The matrix $\eye$ always denotes the identity matrix (its dimension inferred by the context) and for any vector $\vv\in\R^N$, we define $\mD_{\vv}$ as the diagonal matrix with $v_i$ as its $i$-th diagonal entry. Throughout the paper, the terms $\prob(\cdot)$ and $\E[\cdot]$ will stand for probability of an event and expectation of a random variable respectively, and we use $\| \cdot \|_p$ to refer to the $L^p$ norm.

\paragraph{Time-Reversible Markov Chains:} The following definitions and properties are classical in Markov chain literature and will be stated without proof. We point the readers to \citep[Chapter 3.4]{aldous} for a more comprehensive review of the spectral properties stated below.
A Markov chain with kernel $\mP\in[0,1]^{N \times \N}$ is \textit{time-reversible} if there exists a probability distribution $\vmu \in \Sigma$ such that the pair $(\mP,\vmu)$ solve the \textit{detailed balance equation} (DBE), that is, $\mu_iP_{ij} = \mu_jP_{ji}$ for all $i,j \in \cN$. Consequently, $\vmu$ is also the \textit{stationary distribution} or \textit{invariant measure} of $\mP$, and solves $\vmu^T \mP = \vmu^T$. A time-reversible Markov chain's kernel $\mP$ has all eigenvalues on the real line. We denote by $(\lambda_i, \vu_i)$ (by $(\lambda_i,\vv_i)$) its left (right) eigenpair, where the eigenvalues are ordered such that $1 = \lambda_N > \lambda_{N-1} \geq \cdots \lambda_1 \geq -1$, with $\vu_N = \vmu$ and $\vv_N = \ones$. Moreover, we can deduce that $\vu_i = \mD_{\vmu}\vv_i$ and $\vu_i^T \vv_i = 0$ for all $i\in\cN$.

\subsection{SRRW iteration and the related mean-field ODE}

Consider the SRRW on $\cG(\cN,\cE)$ as defined in Section \ref{section:introduction}, where for any ergodic and time-reversible $\mP$ with corresponding stationary measure $\vmu$, and any $\alpha \geq 0$, the function $r_{\mu_i}(\cdot)$ in the non-linear kernel in \eqref{eqn:general srrw kernel} is of the polynomial form as in \eqref{eqn:polynomial_repellence}. In this setup, the transition kernel $\mK[\vx]$ is well-defined for all $\vx \in \Int(\Sigma)$, and we have the following result, whose proof is deferred to Appendix \ref{appendix:algorithmic setup}.

\begin{proposition} \label{prop:form of stationary dist}
For any $\vx \in \Int(\Sigma)$, there exists a unique stationary measure $\vpi(\vx) \triangleq [\pi_i(\vx)]_{i\in\cN} \in \Int(\Sigma)$, where
\begin{equation}\label{eqn:pi_x closed form}
\pi_i(\vx) \propto {\sum \limits_{j\in\cN} \mu_i P_{ij} \left(\frac{x_i}{\mu_i}\right)^{\!-\alpha} \!\left(\frac{x_j}{\mu_j}\right)^{\!-\alpha} }
, ~~\forall ~i\in \cN,
\end{equation}
and the pair $\left(\mK[\vx], \vpi(\vx)\right)$ solves the DBE, that is, $\pi_i(\vx)K_{ij}[\vx] = \pi_j(\vx)K_{ji}[\vx]$ for all $i,j\in\cN$.
\end{proposition}

Since the transition probability $K_{ij}[\vx]$ for any node $i$ to $j$ is only well defined when $x_j > 0$ for each $j\in\cN(i)$, we redefine $\vx_n$ (originally defined in \eqref{eqn:empirical_dist_1}) as $\vx_n \triangleq \frac{1}{n+1}\left[ \vnu + \sum_{k=1}^n \bfdelta_{X_k}\right]$ for some $\vnu \in \Int(\Sigma)$. This redefined sequence of empirical measures $\{\vx_n\}_{n \geq 0}$, which satisfies the recursion
\begin{equation}\label{eqn:srrw iteration original}
    \vx_{n+1} = \vx_n + \frac{1}{n\!+\!2}\left( \bfdelta_{X_{n+1}} - \vx_n \right), ~~~\forall ~n\geq 0,
\end{equation}
starts from $\vx_0 = \vnu \in \Int(\Sigma)$, and satisfies $\vx_n \in \Int(\Sigma)$ which ensures that $\mK[\vx_n]$ is always well defined for every finite $n \geq 0$.

\begin{remark}\label{remark:initial point}
    The S.I. property (ii) along with the form of $K_{ij}[\vx]$ as in \eqref{eqn:general srrw kernel} means that knowing visit counts of each neighboring node of the current position is enough for computing the transition probabilities. This has benefits when the random walker is exploring the graph while performing MCMC sampling on-the-fly. For example, the random walker which associates one `fake' visit to every newly discovered neighboring node automatically ensures that the SRRW iteration begins with $\vnu \propto \ones$ (uniform distribution) without the need to know any \textit{global} graph statistic. This 'fake' visits can also be weighted - assigning it proportional to the degree of the newly discovered neighboring node ensures that $\vnu = \vd$, where $\vd=[\deg(i)/\sum_{k\in\cN} \deg(k)]_{i \in \cN}$ is the degree proportional distribution of the graph. It is also important to note that our main results in Sections \ref{section:stochastic_analysis} are invariant in the choice of $\vnu \in \Int(\Sigma)$. \qed
\end{remark}

The existence and uniqueness of the stationary distribution $\vpi(\vx_n)$ shown in Proposition \ref{prop:form of stationary dist} allows us to further decompose the recursion in \eqref{eqn:srrw iteration original} as
\begin{equation}\label{eqn:srrw iteration}
    \vx_{n+1} = \vx_n + \gamma_{n+1} \left[ \vh(\vx_n)  + \bfepsilon(X_{n+1},\vx_n) \right], 
\end{equation}
where $\gamma_n \triangleq \frac{1}{n+1}$, $\vh(\vx) \triangleq \vpi(\vx) - \vx$ and $\bfepsilon(X,\vx) \triangleq \bfdelta_{X} \!-\! \vpi(\vx)$. The recursion in \eqref{eqn:srrw iteration} is an example of a stochastic approximation (SA) algorithm with controlled Markovian dynamic \citep[see][]{benveniste2012adaptive, Andrieu2005Stability, andrieu2015stability, fort2014central} and decreasing step sizes $\{\gamma_n\}_{n \geq 0}$. For each $n\geq 0$, the random variable $\bfepsilon(X_{n+1},\vx_n)$ captures the \textit{noise} in the updates which, in this case, is driven by the stochastic process $\{X_n\}_{n \geq 0}$. SA algorithms have a strong connection to ordinary differential equations (ODEs) via its \textit{mean-field}, which in our case is $\vh:\Int(\Sigma) \to \R^N$; SA iterates typically converge to $\omega$-limit sets (fixed points and limit cycles) of the semi-flow induced by $\vh(\cdot)$ as long as the noise terms $\{\bfepsilon(X_n,\vx_n)\}_{n \geq 0}$ are shown to have negligible contribution to the SA iterations asymptotically \cite{kushner1997, Benaim99, borkar2008}. Studying the asymptotic properties of the \textit{deterministic} ODE system given by
\begin{equation}\label{eqn:main ode}
    \frac{d}{dt}\vx(t) = \vh(\vx(t)) = \vpi(\vx(t)) - \vx(t)
\end{equation}
are therefore vital to understanding those of the \textit{stochastic} iteration \eqref{eqn:srrw iteration}. For the SRRW kernel as in \eqref{eqn:general srrw kernel}, we have $\mK[\vmu] = \mP$ for any $\alpha \geq 0$ since $r_{\mu_i}(\mu_i) = (\mu_i/\mu_i)^{-\alpha} = 1$ for all $i\in \cN$; and from the balance equation $\vpi(\vmu)^T\mP = \vpi(\vmu)^T$ we have $\vpi(\vmu) = \vmu$. Thus, the target measure $\vmu$ solves $\vpi(\vx) = \vx$, the fixed point equation of \eqref{eqn:main ode}.

\section{Global convergence of the mean-field ODE} \label{section:ode_analysis}

In this section we analyze the asymptotic behaviour of the ODE, specifically the global asymptotic stability of the target distribution $\vmu\in\Int(\Sigma)$. The ODE system in \eqref{eqn:main ode} is positively invariant in $\Int(\Sigma)$, that is, solutions of the system starting in $\Int(\Sigma)$ stay in the same set for all times $t \geq 0$. Indeed, one can check that as $x_i(t) \to 0$ for any $i\in\cN$, the corresponding $\pi_i(\vx) \to 1$, and the $i$-th entry $\frac{dx_i}{dt} = \pi_i(\vx)-x_i$ of \eqref{eqn:main ode} will always be positive. With this, we first provide the following results concerning the fixed points of \eqref{eqn:main ode} in $\Int(\Sigma)$. The proofs of all results in this section are deferred to Appendix \ref{appendix:ode analysis}.

\begin{proposition}\label{prop:uniqueness}
    For any $\alpha \geq 0$, the target stationary measure $\vmu \in \Int(\Sigma)$ is the unique fixed point of ODE \eqref{eqn:main ode}.
\end{proposition}

Our next step is to show convergence of ODE \eqref{eqn:main ode} to its unique fixed point $\vmu$, which will be done via a \textit{Lyapunov function}. For any $\alpha \geq 0$, consider the function $w:\Int(\Sigma) \to \R$ given by
\begin{equation}\label{eqn:lyapunov_function}
    w(\vx) = \sum \limits_{i\in\cK}\sum \limits_{j \in \cN} \mu_i P_{ij} \left(\frac{x_i}{\mu_i}\right)^{\!-\alpha} \!\left(\frac{x_j}{\mu_j}\right)^{\!-\alpha}
\end{equation}
for all $\vx \in \Int(\Sigma)$. The following result shows that $w(\cdot)$ is a \textit{strict} Lyapunov function for ODE \eqref{eqn:main ode}.

\begin{lemma}\label{lemma:strict_lyapunov}
    The function $w:\Int(\Sigma) \to \R$ satisfies $\nabla w(\vx)^T \vh(\vx) \leq 0$ for all $\vx \in \Int(\Sigma)$, with equality only when $\vx = \mu$. Moreover, $w(\vx) \!\to\! \infty$ as $x_i \!\to\! 0$ for any $i\!\in\!\cN$.
\end{lemma}

The Lyapunov result in the form as stated above will be especially vital in Section \ref{section:stochastic_analysis} for proving the first and second order convergence results for the \textit{stochastic} SRRW iteration \eqref{eqn:srrw iteration original}. Our next result, concerning the global convergence of trajectories of \eqref{eqn:main ode}, follows by application of the LaSalle invariance principle \citep[see][Chapter 5]{Vidyasagar} to Lemma \ref{lemma:strict_lyapunov}.

\begin{theorem}\label{thm:global_convergence_ode}
    For any $\alpha \geq 0$, the unique fixed point $\mu$ of ODE \eqref{eqn:main ode} is globally asymptotically stable in $\Int(\Sigma)$; that is, for any initial $\vx(0)\in\Int(\Sigma)$, we have $\vx(t) \to \mu$ as $t \to \infty$.
\end{theorem}

For each $\alpha \geq 0$, let $\mJ(\alpha) \triangleq [J_{ij}(\alpha)]_{i,j\in\cN}$ denote the Jacobian of $\vh(\vx)$ evaluated at $\vx = \vmu$, that is, $J_{ij}(\alpha) \triangleq \frac{\partial h_i(\vx)}{\partial x_j} \Big \vert_{\vx = \vmu}$. We end this section with the following result regarding the form of $\mJ(\alpha)$ and its spectrum in terms of $\alpha$, $\mP$, and the spectrum of $\mP$. This result will be employed to characterize the performance of SRRW in the next section. 

\begin{lemma}\label{lemma:srrw_jacobian}
    For any $\alpha \geq 0$, we have
    \begin{equation}\label{eqn:srrw_jacobian}
        \mJ(\alpha) = 2\alpha \vmu \ones^T - \alpha \mP^T - (\alpha \!+\! 1)\eye.
    \end{equation}
    Furthermore, let $\zeta_i$ denote the $i$-th eigenvalue of $\mJ(\alpha)$. Then, $\zeta_N = -1$, and $\zeta_i = \alpha(-\!1 \!-\!\lambda_{i}) - 1$ for all $i\in \{1,\cdots,N\!-\!1\}$, with $\vu_i$ ($\vv_i$) of $\mP$ now being the right (left) eigenvectors corresponding to $\zeta_i$ for all $i \in \cN$.
\end{lemma}
\section{Main results - Convergence and co-variance ordering of SRRW} \label{section:stochastic_analysis}

In this section, we provide our main results concerning the convergence properties of the SRRW iterate sequence $\{\vx_n\}_{n \geq 0}$ for all $\alpha \geq 0$. First, we prove that starting from any initial $\vx_0 \in \Int(\Sigma)$, the SRRW iterate sequence $\{\vx_n\}_{n \geq 0}$ satisfying \eqref{eqn:srrw iteration} converges almost surely to the target distribution $\vmu$, and then provide a CLT where we characterize the exact form of the arising asymptotic co-variance matrix $\mV(\alpha)$ as a function of $\alpha \geq 0$. We use this to prove our performance ordering as a Loewner ordering of asymptotic co-variance matrices, in the form of \eqref{eqn:covariance_ordering_intro}, as touched upon previously in Section \ref{section:introduction}. We also provide corollaries accompanying our main result for the special case of MCMC sampling, showing that the asymptotic sampling variance decreases in $\alpha\geq 0$ with speed $O(1/\alpha)$. All the following results are proved under the assumption \ref{assu:boundedness} as shown below, which we elaborate later in Remark \ref{remark:boundedness}.

\begin{assumption} \label{assu:boundedness}
    For any $(\vx_0,X_0)\in \Int(\Sigma)\times \cN$, the iterate sequence $\{\vx_n\}_{n \geq 0}$ is $\prob_{X_0,\vx_0}$ - almost surely contained within a compact subset of $\Int(\Sigma)$.
\end{assumption}

\begin{theorem}[Almost sure convergence to target $\vmu$]\label{thm:first order result} Under \ref{assu:boundedness}, $\vx_n$ converges $\prob_{\vx_0,X_0}$-almost surely to $\vmu$ as $n\to\infty$ for any initial $(\vx_0,X_0) \in \Int(\Sigma) \times \cN$ and any $\alpha \geq 0$.
\end{theorem}

\begin{theorem}[Central Limit Theorem]\label{thm:second order result}
Under \ref{assu:boundedness},
\begin{equation} \label{eqn:CLT of iterates}
    \sqrt{n}(\vx_n - \vmu) \xrightarrow[dist.]{n\to \infty} N(\0,\mV(\alpha)),
\end{equation}
for any $\alpha \geq 0$, where $\mV(\alpha) \in \R^{N \times N}$ is given by
\begin{equation} \label{eqn:covariance matrix}
    \mV(\alpha) = \sum_{i=1}^{N-1} \frac{1}{2\alpha (1 + \lambda_i) + 1} \cdot \frac{1+\lambda_i}{1-\lambda_i}\vu_i\vu_i^T.
\end{equation}

\end{theorem}

The explicit form of $\mV(\alpha)$ in Theorem \ref{thm:second order result} is derived by applying Lemma \ref{lemma:srrw_jacobian} in conjunction with a version of \citep[Lemma 6.3.7]{bremaudmarkov2020} modified for our SRRW case. This allows us to fully characterize the co-variance matrix $\mV(\alpha)$, and we have the following performance ordering result as a corollary.

\begin{corollary}[Ordering of asymptotic covariance]\label{cor:avr ordering}
For any $\alpha_1 > \alpha_2 > 0$, we have
$$\mV(\alpha_1) <_L \mV(\alpha_2) <_L \mV(0).$$
\end{corollary}

Corollary \ref{cor:avr ordering} states the asymptotic covariance matrix decrease monotonically (in terms of Loewner ordering) as $\alpha$ increases. Recall that the case $\alpha = 0$ coincides with a purely Markovian random walker with kernel $\mP$, traversing the state space with no self interaction. Thus, the empirical distribution of any SRRW with $\alpha>0$ approximates the target distribution with asymptotically smaller co-variance than the underlying base Markov chain $\mP$. The CLT result in \eqref{eqn:CLT of iterates} also leads to the following corollary on the convergence of $L^p$ norms of the scaled error term $\vx_n-\vmu$.

\begin{corollary}\label{corollary:lp_norm}
    For any $\alpha \geq 0$ and integer $p \in \N$, there exists a constant $C_{p,\alpha} > 0$ such that
    \begin{equation}
        \sqrt{n} \E \left[ \| \vx_n - \mu \|_p \right] \xrightarrow{n\to \infty} C_{p,\alpha}.
    \end{equation}
\end{corollary}

\begin{remark} \label{remark:boundedness} \ref{assu:boundedness} assumes the \textit{stability} of the SRRW iterates $\{\vx_n\}_{n \geq 0}$, and is a commonplace in convergence analysis of SA algorithms. Proving \ref{assu:boundedness} for a given SA algorithm is highly non-trivial in most practical applications, and can be especially difficult for the case of SA with state-dependent Markovian noise, where the state space is an open subset of an $\R^N$ \cite{andrieu2015stability,karmakar2020stochastic} - the case with our SRRW process. The issue of proving stability is often dealt with by introducing modifications to the original sequence, such as random restarts of the algorithm as the iterates exit a growing sequence of compact subsets of the open state space. In Appendix \ref{appendix:discussion on boundedness}, we use one such modification as a tool to provide intuition behind why \ref{assu:boundedness} is very likely to be satisfied, even though the formal proof remains an open problem. The critical logic behind this intuition revolves around the uniqueness of the target measure $\vmu$ and Lemma \ref{lemma:strict_lyapunov} regarding the Lyapunov function, and is further supported by our numerical simulations in Section \ref{section:numerical_results}. \qed 
\end{remark}

\paragraph{SRRW for MCMC sampling:}

We show that the performance improvement alluded by the form of $\mV(\alpha)$ from Theorem \ref{thm:second order result} and Corollary \ref{cor:avr ordering} is realized when the SRRW is used for MCMC sampling. Consider a sampling agent following the SRRW process for an appropriate choice     of $\mP$ and $\vmu$, for any given $\alpha \geq 0$ on a general graph $\cG(\cN,\cE)$. At each time step $n>0$, it records a sample according to its current location $X_n$, given by $g(X_n)$ for some scalar function $g:\cN \to \R$, whose vectorized form we denote by $\vg \triangleq [g(i)]_{i \in \cN} \in \R^N$, and updates the MCMC \textit{estimator} $\psi_n(g) \triangleq \frac{1}{n}\sum_{k=1}^n g(X_k)$ with the latest sample. Its goal is to estimate the quantity $\vg^T\vmu = \E_{X \sim \mu}[g(X)]$. We now have the following: 

\begin{corollary}\label{cor:estimator and mse clt}
For any scalar valued function $g:\cN \to \R$ such that $\max_{i\in\cN} |g(i)| < \infty$, and for any $\alpha \geq 0$, we have
\begin{align}
    \psi_n(g) &\xrightarrow[a.s.]{n \to \infty} \vg^T\vmu, \\
    \sqrt{n}(\psi_n(g) - \vg^T\vmu) &\xrightarrow[dist.]{n \to \infty} N(0, \vg^T\mV(\alpha)\vg).
\end{align}
\end{corollary}

\begin{corollary}\label{cor:avr bound}
    Under the same assumption of Corollary \ref{cor:estimator and mse clt}, we have
    \begin{equation}\label{eqn:reduction ratio}
        \frac{\vg^T\mV(\alpha)\vg}{\vg^T\mV(0)\vg} \leq \E \left[ \frac{1}{2\alpha(1 + \Lambda) + 1} \right], 
    \end{equation}
    where $\Lambda \!\in\! (\!-1,1)$ is a random variable taking values $\lambda_i, ~i \!\in\! \{1,\!\cdots\!,N\!-\!1\}$ with probability proportional to $(\vg^T \vu_i)^2$.
\end{corollary}

Corollary \ref{cor:estimator and mse clt} provides almost sure convergence and CLT for the SRRW-driven MCMC estimator $\psi_n(g)$, with its asymptotic variance in terms of $\vg$ and $\mV(\alpha)$. Corollary \ref{cor:avr bound} then quantifies the improvement of the SRRW-driven sampler with $\alpha > 0$ over the baseline MCMC-driven by $\mP$ ($\alpha = 0$ case). We see that reduction ratio goes down to zero as $\alpha \to \infty$ with speed $O(1/\alpha)$. Consider an \textit{i.i.d.} random variable $Y$ with distribution $\vmu$ for which the sampling variance is given by $\text{Var}(g(Y)) = (\vg^2)^T\vmu -  (\vg^T\vmu)^2 > 0$ for all non-trivial choices of function $g(\cdot)$ and $\mu$.\footnote{Here, $\vg^2$ is the vector with its $i$'th entry as $g(i)^2$.} Since $\vg^T\mV(\alpha) \vg \to 0$ as $\alpha \to \infty$, there exists some $\hat \alpha > 0$ such that for each $\alpha > \hat \alpha$, the SRRW as an MCMC estimator outperforms even an i.i.d. sampler, asymptotically. While the value of $\hat \alpha$ depends on the function $g(\cdot)$, target measure $\vmu$ and the spectrum of the underlying base chain $\mP$, we expect that this crossover would take place for moderate values of $\alpha$, as seen from \eqref{eqn:reduction ratio} with (at least) $O(1/\alpha)$ speed. This has the advantage of offsetting potentially undesirable transient behaviour resulting from employing very large $\alpha$, as will also be elaborated in the next section while presenting our simulation results.

Before ending this section, we briefly show that for any choice of baseline Markov chain $\mP$ with stationary distribution $\vmu$, one can construct an \textit{unbiased} version of the estimator $\psi_n(g)$ which estimates the \textit{uniform} average of the function $g(\cdot)$, given by $\frac{1}{N}\vg^T\ones = \E_{X\sim\text{unif.}}[g(X)]$. This procedure is virtually identical to the importance-reweighting typically done in MCMC applications \cite{LeeSIGMETRICS12}, but now sampled by SRRW instead of a Markov chain. To be specific, consider a \textit{weight function} $w:\cN \to \R$ where $w(i) \propto \frac{1}{\mu_i}$ for all $i\in\cN$, and let $h:\cN \to \R$ be $h(i) \triangleq w(i)g(i)$ with $\vh \triangleq [h(i)]_{i\in\cN} \in \R^N$ being its vectorized form.
Consider a new estimator, defined as
\begin{equation}\label{eqn:_importance_rw_unbiased_esimator}
    \hat \psi_n(g) \triangleq \frac{\psi_n(wg)}{\psi_n(w)} = \frac{\sum_{k=1}^n w(X_k)g(X_k)}{\sum_{k=1}^n w(X_k)}.
\end{equation}
We then the following corollary, the proof of which follows by application of Slutsky's theorem \citep[][pg. 332]{ash2000probability}.
\begin{corollary}\label{cor:importance_reweighting}
For the estimator $\hat \psi_n(g)$, results of Corollaries \ref{cor:estimator and mse clt} and \ref{cor:avr bound} hold with $\vg$ and $\vmu$ replaced by $\vh$ and $\frac{1}{N}\ones$ respectively.
\end{corollary}

\section{Simulation results} \label{section:numerical_results}

\begin{figure*}[!ht]
    \centering
    \subfigure[Convergence of $\vx_n$ to the uniform distribution.]{\includegraphics[width=0.5\textwidth]{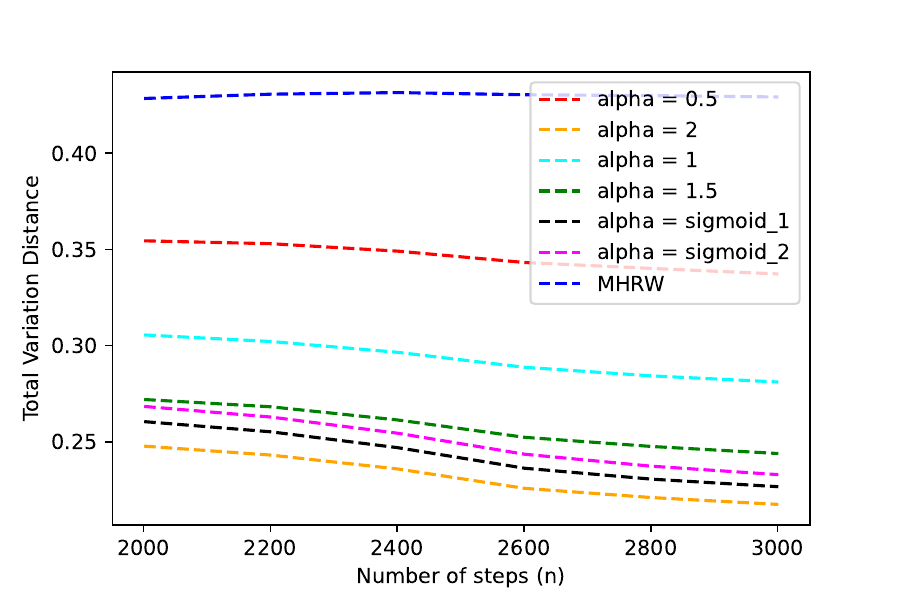} \label{fig:1a}} \hspace{-10mm}
    \subfigure[Convergence of $\psi_n(g)$ to the ground truth $\vg^T\ones/N$.]{\includegraphics[width=0.5\textwidth]{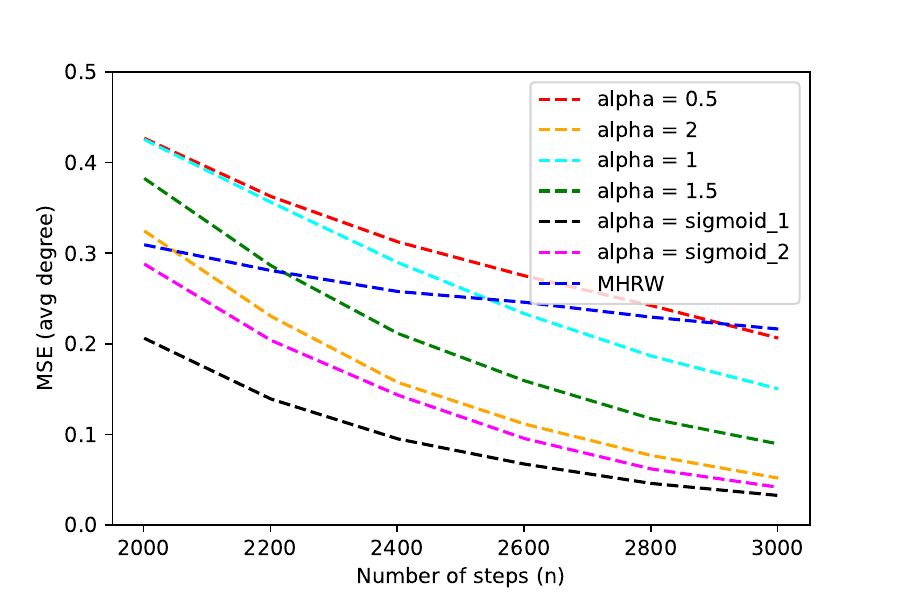} \label{fig:1b}}
    \caption{Simulations of the SRRW process for values of $\alpha \in [0,2]$, where $\alpha = 0$ corresponds to \textit{MHRW} - the underlying Metropolis-Hastings \textit{base} chain, with no self-repellence properties. The two \textit{sigmoid} functions refer to the case where $\alpha$ is made to gradually increase over time, from $0$ to $2$. Sigmoid-1 is of the type $\frac{1}{0.5+e^{-n+0.25 N}}$ while Sigmoid-2 is of the type $\frac{n}{100+0.5n}$, where $n\geq0$ is the time parameter. Further tuning of the sigmoid functions may lead to empirically more efficient MCMC algorithms.}
    \label{fig:experiments}
\end{figure*}

We now present simulation results which support our theoretical findings. We simulate the SRRW process over the wiki-Vote graph \citep{snap}, which is an undirected, connected graph with $889$ nodes and $2914$ edge. We set $\mP$, the \textit{base} Markov chain, to be the Metropolis Hastings random walk (MHRW) - a commonly used Markov chain for unbiased graph sampling. Our target measure $\vmu$ is therefore the uniform distribution, that is, $\vmu = \frac{1}{N}\ones$. We also measure the \textit{average degree} of the graph by employing the estimator $\psi_n(g)$, where $g(i) = \text{deg}(i)$ for each $i\in\cN$, and $\vg^T\vmu = \sum_{i\in\cN}\text{deg(i)}/N$ .

To track the convergence of the stochastic process, we measure the \textit{Total Variation Distance} (TVD), evaluated at each time step as
\begin{equation}
    \text{TVD}(\vx_n,\vmu) = \frac{1}{2}\| \vx_n - \vmu \|,
\end{equation}
where $\vx_n$ is the empirical distribution of the SRRW process. The TVD is then averaged over the total number of simulation runs. We also keep track of the \textit{Mean Square Error} (MSE) of the estimator $\psi_n(g)$ in estimating the average degree, given by
\begin{equation}
    \text{MSE}(\psi_n(g),\vg^T\vmu) = \frac{1}{K}\sum_{k=1}^K (\psi^k_n(g) - \vg^T\vmu)^2
\end{equation}
where the \textit{mean} is taken over the total number of simulation runs $K$, and $\psi_n^k(g)$ denotes the estimator for the $k$-th simulation run. 

Our final simulation results are shown in Figures \ref{fig:1a} and \ref{fig:1b}, where the curves are obtained as averages over $K=500$ simulation runs. For both, TVD and MSE, we observe that larger values of $\alpha$ perform better asymptotically as predicted by our theory, significantly outperforming the case of $\alpha=0$ (the sampler driven by base chain MHRW). We also observe for the MSE plot in Figure \ref{fig:1b} that the case of time-varying $\alpha$ provides better performance for the task of estimating the average degree. While our theoretical framework only supports constant values of $\alpha\geq 0$, the numerically observed convergence and efficiency of the time-varying case opens up possibilities for construction of a more adaptive version of the SRRW algorithm, which can be a possible future direction. Additional simulation results comparing Self-Repellent Random walks with non-backtracking approaches can be found in Appendix \ref{appendix:additional_simulation_results}.

\section{Conclusion}\label{section:conclusion}

In this paper, we introduced a \textit{self-repellent random walk} which can be designed to sample from any target distribution $\vmu\in\Int(\Sigma)$ via nonlinear interacting with its own occupational measure. We provide convergence results for the sequence of occupational measures by first analyzing a closely related deterministic process whose convergence is proved with the aid of a Lyapunov function, and then utilizing results from stochastic approximation theory to establish first and second order convergence results. We provide an explicit form for the asymptotic co-variance matrix arising out of the CLT result, which enables us to prove that the sampling variance decreases monotonically to zero as $\alpha\to \infty$, with speed at least $O(1/\alpha)$. Our results advocate for study into the design of nonlinear Markov chains for MCMC and other learning applications.

\section{Acknowledgments and Disclosure of Funding}
We thank the anonymous reviewers for their constructive comments. This research was largely conducted while Vishwaraj Doshi was with the Operations Research Graduate Program, North Carolina State University. This work was supported in part by National Science Foundation under Grant Nos. CNS-2007423, IIS-1910749, and CNS-1824518.

\bibliography{references_master.bib}

\begin{thebibliography}{56}
\providecommand{\natexlab}[1]{#1}
\providecommand{\url}[1]{\texttt{#1}}
\expandafter\ifx\csname urlstyle\endcsname\relax
  \providecommand{\doi}[1]{doi: #1}\else
  \providecommand{\doi}{doi: \begingroup \urlstyle{rm}\Url}\fi

\bibitem[Akian et~al.(2007)Akian, Gaubert, and Ninove]{akian2007multiple}
Akian, M., Gaubert, S., and Ninove, L.
\newblock Multiple equilibria of nonhomogeneous markov chains and
  self-validating web rankings.
\newblock \emph{arXiv preprint arXiv:0712.0469}, 2007.

\bibitem[Aldous \& Fill(2002)Aldous and Fill]{aldous}
Aldous, D. and Fill, J.~A.
\newblock Reversible markov chains and random walks on graphs, 2002.
\newblock Unfinished monograph, recompiled 2014, available at
  \url{http://www.stat.berkeley.edu/~aldous/RWG/book.html}.

\bibitem[Alon et~al.(2007)Alon, Benjamini, Lubetzky, and Sodin]{alon2007non}
Alon, N., Benjamini, I., Lubetzky, E., and Sodin, S.
\newblock Non-backtracking random walks mix faster.
\newblock \emph{Communications in Contemporary Mathematics}, 9\penalty0
  (04):\penalty0 585--603, 2007.

\bibitem[Amit et~al.(1983)Amit, Parisi, and Peliti]{Amit1983asymptotic}
Amit, D.~J., Parisi, G., and Peliti, L.
\newblock Asymptotic behavior of the "true" self-avoiding walk.
\newblock \emph{Phys. Rev. B}, 27:\penalty0 1635--1645, Feb 1983.

\bibitem[Andrieu \& Livingstone(2021)Andrieu and
  Livingstone]{andrieu2021peskun}
Andrieu, C. and Livingstone, S.
\newblock Peskun--tierney ordering for markovian monte carlo: Beyond the
  reversible scenario.
\newblock \emph{The Annals of Statistics}, 49\penalty0 (4):\penalty0
  1958--1981, 2021.

\bibitem[Andrieu et~al.(2005)Andrieu, Moulines, and
  Priouret]{Andrieu2005Stability}
Andrieu, C., Moulines, {\'E}., and Priouret, P.
\newblock Stability of stochastic approximation under verifiable conditions.
\newblock \emph{SIAM Journal on control and optimization}, 44\penalty0
  (1):\penalty0 283--312, 2005.

\bibitem[Andrieu et~al.(2007{\natexlab{a}})Andrieu, Jasra, Doucet, and
  Del~Moral]{Andrieu2007Nonlinear}
Andrieu, C., Jasra, A., Doucet, A., and Del~Moral, P.
\newblock Non-linear markov chain monte carlo.
\newblock \emph{Esaim: Proceedings}, 19:\penalty0 79--84, 01
  2007{\natexlab{a}}.

\bibitem[Andrieu et~al.(2007{\natexlab{b}})Andrieu, Jasra, Doucet, and
  Del~Moral]{andrieu2007convergence}
Andrieu, C., Jasra, A., Doucet, A., and Del~Moral, P.
\newblock Convergence of the equi-energy sampler.
\newblock In \emph{ESAIM: Proceedings}, volume~19, pp.\  1--5. EDP Sciences,
  2007{\natexlab{b}}.

\bibitem[Andrieu et~al.(2008)Andrieu, Jasra, Doucet, and
  Del~Moral]{andrieu2008note}
Andrieu, C., Jasra, A., Doucet, A., and Del~Moral, P.
\newblock A note on convergence of the equi-energy sampler.
\newblock \emph{Stochastic Analysis and Applications}, 26\penalty0
  (2):\penalty0 298--312, 2008.

\bibitem[Andrieu et~al.(2011)Andrieu, Jasra, Doucet, and
  Moral]{Andrieu2011OnNonlinearMCMC}
Andrieu, C., Jasra, A., Doucet, A., and Moral, P.~D.
\newblock {On nonlinear Markov chain Monte Carlo}.
\newblock \emph{Bernoulli}, 17\penalty0 (3):\penalty0 987 -- 1014, 2011.

\bibitem[Andrieu et~al.(2015)Andrieu, Tadi{\'c}, and
  Vihola]{andrieu2015stability}
Andrieu, C., Tadi{\'c}, V.~B., and Vihola, M.
\newblock On the stability of some controlled markov chains and its
  applications to stochastic approximation with markovian dynamic.
\newblock \emph{The Annals of Applied Probability}, 25\penalty0 (1):\penalty0
  1--45, 2015.

\bibitem[Angel et~al.(2014)Angel, Crawford, and Kozma]{angel2014localization}
Angel, O., Crawford, N., and Kozma, G.
\newblock Localization for linearly edge reinforced random walks.
\newblock \emph{Duke Mathematical Journal}, 163\penalty0 (5):\penalty0
  889--921, 2014.

\bibitem[Ash \& Dol{\'e}ans-Dade(2000)Ash and
  Dol{\'e}ans-Dade]{ash2000probability}
Ash, R.~B. and Dol{\'e}ans-Dade, C.~A.
\newblock \emph{Probability and measure theory}.
\newblock Academic press, 2000.

\bibitem[Avrachenkov et~al.(2021)Avrachenkov, Borkar, Moharir, and
  Shah]{avrachenkov2021dynamic}
Avrachenkov, K.~E., Borkar, V.~S., Moharir, S., and Shah, S.~M.
\newblock Dynamic social learning under graph constraints.
\newblock \emph{IEEE Transactions on Control of Network Systems}, 9\penalty0
  (3):\penalty0 1435--1446, 2021.

\bibitem[Ball(2016)]{JB_semiflow}
Ball, J.
\newblock {Semiflows, Lyapunov functions and approach to equilibrium}.
\newblock University lecture, 2016.
\newblock URL
  \url{https://people.maths.ox.ac.uk/ball/Teaching/cdtsemiflows16.pdf}.

\bibitem[Bena{\"i}m(1997)]{Benaim97vertex}
Bena{\"i}m, M.
\newblock Vertex-reinforced random walks and a conjecture of pemantle.
\newblock \emph{Annals of Probability}, 25\penalty0 (1):\penalty0 361--392, 01
  1997.

\bibitem[Bena{\"i}m(1999)]{Benaim99}
Bena{\"i}m, M.
\newblock \emph{Dynamics of stochastic approximation algorithms}, pp.\  1--68.
\newblock Springer Berlin Heidelberg, 1999.

\bibitem[Bena{\"\i}m \& Tarres(2011)Bena{\"\i}m and Tarres]{benaim2011dynamics}
Bena{\"\i}m, M. and Tarres, P.
\newblock Dynamics of vertex-reinforced random walks.
\newblock \emph{The Annals of Probability}, 39\penalty0 (6):\penalty0
  2178--2223, 2011.

\bibitem[Bena{\"i}m et~al.(2012)Bena{\"i}m, Raimond, and
  Schapira]{benaim2012strongly}
Bena{\"i}m, M., Raimond, O., and Schapira, B.
\newblock Strongly vertex-reinforced-random-walk on the complete graph.
\newblock \emph{arXiv preprint arXiv:1208.6375}, 2012.

\bibitem[Benveniste et~al.(2012)Benveniste, M{\'e}tivier, and
  Priouret]{benveniste2012adaptive}
Benveniste, A., M{\'e}tivier, M., and Priouret, P.
\newblock \emph{Adaptive algorithms and stochastic approximations}, volume~22.
\newblock Springer Science \& Business Media, 2012.

\bibitem[Borkar(2008)]{borkar2008}
Borkar, V.~S.
\newblock \emph{Stochastic Approximation - A Dynamical Systems Viewpoint}.
\newblock Hindustan Book Agency, 2008.

\bibitem[Br{\'e}maud(2020)]{bremaudmarkov2020}
Br{\'e}maud, P.
\newblock Markov chains gibbs fields, monte carlo simulation, and queues.
\newblock 2020.

\bibitem[Chen(2014)]{Chen2014Two}
Chen, J.
\newblock {Two particles' repelling random walks on the complete graph}.
\newblock \emph{Electronic Journal of Probability}, 19\penalty0
  (none):\penalty0 1 -- 17, 2014.

\bibitem[Chen \& Hwang(2013)Chen and Hwang]{chen2013accelerating}
Chen, T.-L. and Hwang, C.-R.
\newblock Accelerating reversible markov chains.
\newblock \emph{Statistics \& Probability Letters}, 83\penalty0 (9):\penalty0
  1956--1962, 2013.

\bibitem[Del~Moral \& Miclo(2004)Del~Moral and Miclo]{Moral2004OnConvergence}
Del~Moral, P. and Miclo, L.
\newblock On convergence of chains with occupational self-interactions.
\newblock \emph{Proceedings of the Royal Society of London. Series A:
  Mathematical, Physical and Engineering Sciences}, 460\penalty0
  (2041):\penalty0 325--346, 2004.

\bibitem[Del~Moral \& Miclo(2006)Del~Moral and Miclo]{Moral2006SelfInteracting}
Del~Moral, P. and Miclo, L.
\newblock Self-interacting markov chains.
\newblock \emph{Stochastic Analysis and Applications}, 24:\penalty0 615--660,
  07 2006.

\bibitem[Delyon(2000)]{delyon2000stochastic}
Delyon, B.
\newblock Stochastic approximation with decreasing gain: Convergence and
  asymptotic theory.
\newblock Technical report, Universit{\'e} de Rennes, 2000.

\bibitem[Diaconis et~al.(2000)Diaconis, Holmes, and Neal]{diaconis2000analysis}
Diaconis, P., Holmes, S., and Neal, R.~M.
\newblock Analysis of a nonreversible markov chain sampler.
\newblock \emph{Annals of Applied Probability}, pp.\  726--752, 2000.

\bibitem[Fort(2015)]{fort2015central}
Fort, G.
\newblock Central limit theorems for stochastic approximation with controlled
  markov chain dynamics.
\newblock \emph{ESAIM: Probability and Statistics}, 19:\penalty0 60--80, 2015.

\bibitem[Fort et~al.(2011)Fort, Moulines, and Priouret]{fort2011convergence}
Fort, G., Moulines, E., and Priouret, P.
\newblock Convergence of adaptive and interacting markov chain monte carlo
  algorithms.
\newblock \emph{The Annals of Statistics}, 39\penalty0 (6):\penalty0
  3262--3289, 2011.

\bibitem[Fort et~al.(2014)Fort, Moulines, Priouret, and
  Vandekerkhove]{fort2014central}
Fort, G., Moulines, E., Priouret, P., and Vandekerkhove, P.
\newblock A central limit theorem for adaptive and interacting markov chains.
\newblock \emph{Bernoulli}, 20\penalty0 (2):\penalty0 457--485, 2014.

\bibitem[Fortuin et~al.(1971)Fortuin, Kasteleyn, and
  Ginibre]{fortuin1971correlation}
Fortuin, C.~M., Kasteleyn, P.~W., and Ginibre, J.
\newblock Correlation inequalities on some partially ordered sets.
\newblock \emph{Communications in Mathematical Physics}, 22:\penalty0 89--103,
  1971.

\bibitem[Grassberger(2017)]{grassberger2017self}
Grassberger, P.
\newblock Self-trapping self-repelling random walks.
\newblock \emph{Physical review letters}, 119\penalty0 (14):\penalty0 140601,
  2017.

\bibitem[Harold J.~Kushner(1997)]{kushner1997}
Harold J.~Kushner, G. G.~Y.
\newblock \emph{Stochastic Approximation Algorithms and Applications}.
\newblock Springer, 1997.

\bibitem[Hu et~al.(2022)Hu, Doshi, and Eun]{hu2022efficiency_arxiv}
Hu, J., Doshi, V., and Eun, D.~Y.
\newblock Efficiency ordering of stochastic gradient descent.
\newblock \emph{arXiv preprint arXiv:2209.07446}, 2022.

\bibitem[Karmakar(2020)]{karmakar2020stochastic}
Karmakar, P.
\newblock Stochastic approximation with markov noise: Analysis and applications
  in reinforcement learning.
\newblock \emph{arXiv preprint arXiv:2012.00805}, 2020.

\bibitem[Lee et~al.(2012)Lee, Xu, and Eun]{LeeSIGMETRICS12}
Lee, C.-H., Xu, X., and Eun, D.~Y.
\newblock {Beyond Random Walk and Metropolis-Hastings Samplers: Why You Should
  Not Backtrack for Unbiased Graph Sampling}.
\newblock In \emph{Proceedings of the ACM SIGMETRICS/PERFORMANCE Joint
  International Conference on Measurement and Modeling of Computer Systems},
  SIGMETRICS'12, pp.\  319--330, 2012.

\bibitem[Leskovec \& Krevl(2014)Leskovec and Krevl]{snap}
Leskovec, J. and Krevl, A.
\newblock {SNAP Datasets}: {Stanford} large network dataset collection.
\newblock \url{http://snap.stanford.edu/data}, June 2014.

\bibitem[Li et~al.(2015)Li, Yu, Qin, Mao, and Jin]{li2015random}
Li, R.-H., Yu, J.~X., Qin, L., Mao, R., and Jin, T.
\newblock On random walk based graph sampling.
\newblock In \emph{2015 IEEE 31st international conference on data
  engineering}, pp.\  927--938. IEEE, 2015.

\bibitem[Ma et~al.(2016)Ma, Chen, Wu, and Fox]{ma2016unifying}
Ma, Y.-A., Chen, T., Wu, L., and Fox, E.~B.
\newblock A unifying framework for devising efficient and irreversible mcmc
  samplers.
\newblock \emph{arXiv preprint arXiv:1608.05973}, 2016.

\bibitem[Moral \& Doucet(2010)Moral and Doucet]{Moral2010Interacting}
Moral, P.~D. and Doucet, A.
\newblock {Interacting Markov chain Monte Carlo methods for solving nonlinear
  measure-valued equations}.
\newblock \emph{The Annals of Applied Probability}, 20\penalty0 (2):\penalty0
  593 -- 639, 2010.

\bibitem[Neal(2004)]{neal2004improving}
Neal, R.~M.
\newblock Improving asymptotic variance of mcmc estimators: Non-reversible
  chains are better.
\newblock \emph{arXiv preprint math/0407281}, 2004.

\bibitem[Pemantle(1988)]{pemantle1988phase}
Pemantle, R.
\newblock Phase transition in reinforced random walk and rwre on trees.
\newblock \emph{The Annals of Probability}, pp.\  1229--1241, 1988.

\bibitem[Pemantle(2007)]{pemantle2007survey}
Pemantle, R.
\newblock A survey of random processes with reinforcement.
\newblock \emph{Probability surveys}, 4:\penalty0 1--79, 2007.

\bibitem[Rosales et~al.(2022)Rosales, Prado, and Pires]{rosales2022vertex}
Rosales, R.~A., Prado, F.~P., and Pires, B.
\newblock Vertex reinforced random walks with exponential interaction on
  complete graphs.
\newblock \emph{Stochastic Processes and their Applications}, 148:\penalty0
  353--379, 2022.

\bibitem[Slotine \& Li(1991)Slotine and Li]{Slotine}
Slotine, J.-J.~E. and Li, W.
\newblock \emph{Applied Nonlinear Control}.
\newblock Prentice-Hall, 1991.

\bibitem[Stevens \& Othmer(1997)Stevens and Othmer]{stevens1997aggregation}
Stevens, A. and Othmer, H.~G.
\newblock Aggregation, blowup, and collapse: the abc's of taxis in reinforced
  random walks.
\newblock \emph{SIAM Journal on Applied Mathematics}, 57\penalty0 (4):\penalty0
  1044--1081, 1997.

\bibitem[Sun et~al.(2018)Sun, Sun, and Yin]{sun2018on}
Sun, T., Sun, Y., and Yin, W.
\newblock On markov chain gradient descent.
\newblock \emph{Advances in neural information processing systems}, 31, 2018.

\bibitem[Tarr{\`e}s(2004)]{tarres2004vertex}
Tarr{\`e}s, P.
\newblock {Vertex-reinforced random walk on Z eventually gets stuck on five
  points}.
\newblock \emph{The Annals of Probability}, 32\penalty0 (3B):\penalty0
  2650--2701, 2004.

\bibitem[Thin et~al.(2020)Thin, Kotelevskii, Andrieu, Durmus, Moulines, and
  Panov]{thin2020nonreversible}
Thin, A., Kotelevskii, N., Andrieu, C., Durmus, A., Moulines, E., and Panov, M.
\newblock Nonreversible mcmc from conditional invertible transforms: a complete
  recipe with convergence guarantees.
\newblock \emph{arXiv preprint arXiv:2012.15550}, 2020.

\bibitem[Toth(1995)]{Toth1995bond}
Toth, B.
\newblock {The "True" Self-Avoiding Walk with Bond Repulsion on $\mathbb{Z}$:
  Limit Theorems}.
\newblock \emph{The Annals of Probability}, 23\penalty0 (4):\penalty0 1523 --
  1556, 1995.

\bibitem[Turitsyn et~al.(2011)Turitsyn, Chertkov, and
  Vucelja]{turitsyn2011irreversible}
Turitsyn, K.~S., Chertkov, M., and Vucelja, M.
\newblock Irreversible monte carlo algorithms for efficient sampling.
\newblock \emph{Physica D: Nonlinear Phenomena}, 240\penalty0 (4-5):\penalty0
  410--414, 2011.

\bibitem[Van~der Vaart(2000)]{van2000asymptotic}
Van~der Vaart, A.~W.
\newblock \emph{Asymptotic statistics}, volume~3.
\newblock Cambridge university press, 2000.

\bibitem[Veto \& Toth(2008)Veto and Toth]{veto2008self}
Veto, B. and Toth, B.
\newblock {Self-repelling random walk with directed edges on Z}.
\newblock \emph{Electronic Journal of Probability}, 13\penalty0
  (none):\penalty0 1909 -- 1926, 2008.

\bibitem[Vidyasagar(1993)]{Vidyasagar}
Vidyasagar, M.
\newblock \emph{Nonlinear Systems Analysis}.
\newblock Prentice-Hall, Englewood Cliffs, NJ, USA, 1993.

\bibitem[Volkov(2006)]{volkov2006phase}
Volkov, S.
\newblock Phase transition in vertex-reinforced random walks on z with
  non-linear reinforcement.
\newblock \emph{J. Theoret. Probab}, 19\penalty0 (3):\penalty0 691--700, 2006.

\end{thebibliography}
\bibliographystyle{icml2023}

\newpage
\appendix
\onecolumn


\section{Proof of results in Section \ref{Section:algorithmic setup}}\label{appendix:algorithmic setup}

\begin{proof}[\textbf{Proof of Proposition \ref{prop:form of stationary dist}}]
    It is enough to show that the form of $\vpi(\vx)$ specified in \eqref{eqn:pi_x closed form} satisfies the detailed balance equation. For any $i,j\in\cN$, we have 
    \begin{align}
        \pi_i(\vx)K_{ij}[\vx] &\propto \left[\sum_{k\in\cN}\mu_i P_{ik} \left( \frac{x_i}{\mu_i} \right)^{-\alpha} \left( \frac{x_k}{\mu_k} \right)^{-\alpha}\right] \frac{P_{ij}\left( \frac{x_j}{\mu_j} \right)^{-\alpha}}{\sum_{k\in\cN} P_{ik} \left( \frac{x_k}{\mu_k} \right)^{-\alpha}} \times \frac{\mu_i}{\mu_i} \nonumber  \\
        &=  \left( \frac{x_i}{\mu_i} \right)^{-\alpha}\left[\sum_{k\in\cN}\mu_i P_{ik}  \left( \frac{x_k}{\mu_k} \right)^{-\alpha}\right] \frac{\mu_iP_{ij}\left( \frac{x_j}{\mu_j} \right)^{-\alpha}}{\sum_{k\in\cN} \mu_iP_{ik} \left( \frac{x_k}{\mu_k} \right)^{-\alpha}} \nonumber\\
        &= \mu_iP_{ij}\left( \frac{x_i}{\mu_i} \right)^{-\alpha}\left( \frac{x_j}{\mu_j} \right)^{-\alpha}. \label{eqn:srrw_reversible_1}
    \end{align}
    Similarly, we can write
    \begin{equation}\label{eqn:srrw_reversible_2}
        \pi_j(\vx)K_{ji}[\vx] \propto \mu_jP_{ji}\left( \frac{x_i}{\mu_i} \right)^{-\alpha}\left( \frac{x_j}{\mu_j} \right)^{-\alpha},
    \end{equation}
    and using the fact that $\mu_iP_{ij} = \mu_jP_{ji}$ due to the time-reversibility property of $\mP$ with $\vmu$, \eqref{eqn:srrw_reversible_1} and \eqref{eqn:srrw_reversible_2} are equivalent, which completes the proof.
\end{proof}

\section{Proof of results in Section \ref{section:ode_analysis}}\label{appendix:ode analysis}

\begin{proof}[\textbf{Proof of Proposition \ref{prop:uniqueness}}]
    We begin by using the form of $\vpi(\vx)$ derived in Proposition \ref{prop:form of stationary dist} to study the fixed point equation for ODE \eqref{eqn:main ode}, which can be written for each $i \in \cN$ as
    \begin{align*}
        x_i = \pi_i(\vx) = \frac{1}{D}{\sum \limits_{j\in\cN} \mu_i P_{ij} \left(\frac{x_i}{\mu_i}\right)^{\!-\alpha} \!\left(\frac{x_j}{\mu_j}\right)^{\!-\alpha}},
    \end{align*}
    where the normalizing constant $D$ is given by $D = \sum \limits_{i\in\cN} \sum \limits_{j\in\cN} \mu_i P_{ij} \left(\frac{x_i}{\mu_i}\right)^{\!-\alpha} \!\left(\frac{x_j}{\mu_j}\right)^{\!-\alpha}$. Consider a change of variable where $y_i \triangleq x_i/\mu_i$ for all $i\in\cN$, and rewrite the above equation to obtain
    \begin{equation}\label{eqn:fixed_point_3}
        y_i = \frac{1}{D} y_i^{-\alpha} \sum_{j\in\cN} P_{ij} y_j^{-\alpha}.
    \end{equation}
    It is enough to show that the only possible positive solutions to \eqref{eqn:fixed_point_3} are of the type $y_i = c > 0$ for all $i\in\cN$. To check that it is indeed a solution, observe that
    \begin{align*}
        D &= c^{-2\alpha} \sum_{i\in\cN} \sum_{j\in\cN} \mu_iP_{ij}
          = c^{-2\alpha} \sum_{i\in\cN} \sum_{j\in\cN} \mu_jP_{ji} 
          = c^{-2\alpha}\sum_{j\in\cN} \mu_j \sum_{i\in\cN}P_{ji}= c^{-2\alpha},
    \end{align*}
    where the second equality comes from the time-reversibility of $(\mu,\mP)$, which means that $\mu_i P_{ij} = \mu_j P_{ji}$ for all $i,j\in\cN$, and the fourth equality comes from $\mP$ being a stochastic matrix. Substituting $D = c^{-2\alpha}$ in \eqref{eqn:fixed_point_3} gives us
    \begin{align*}
        y_i = c = \frac{1}{c^{-2\alpha}} c^{-\alpha}\sum \limits_{j \in \cN} P_{ij} c^{-\alpha} = 1,
    \end{align*}
    implying $x_i = \mu_i$ for all $i\in\cN$.

    Suppose there exists another positive solution $\vy'$ of \eqref{eqn:fixed_point_3} such that $\vy' \neq c\ones$ for some $c>0$. Without loss of generality, we assume $y_1' = \max_{i\in\cN} \{y_i'\}$, $y_N' = \min_{i\in\cN}\{y_i'\}$. Since $\vy' \neq c\ones$, we have $y_1' > y_N'$, and from the fixed point equation \eqref{eqn:fixed_point_3} we have
    \begin{subequations}
    \begin{equation}\label{eqn:y1'}
        y_1' = \frac{1}{D} (y_1')^{-\alpha} \sum_{j\in\cN} P_{1j} (y_j')^{-\alpha},
    \end{equation}
    \begin{equation}\label{eqn:yN'}
        y_N' = \frac{1}{D} (y_N')^{-\alpha} \sum_{k\in\cN} P_{Nk} (y_k')^{-\alpha}.
    \end{equation}
    \end{subequations}
    Dividing \eqref{eqn:y1'} by \eqref{eqn:yN'} and rearranging some terms gives us
    \begin{equation}\label{eqn:y1/yN}
        \left(\frac{y_1'}{y_N'}\right)^{\alpha+1} = \frac{\sum_{j\in\cN} P_{1j} y_j'^{-\alpha}}{\sum_{k\in\cN} P_{Nk} y_k'^{-\alpha}}.
    \end{equation}
    
    Since $\alpha \geq 0$, we have $(y_1')^{-\alpha} \leq (y_i')^{-\alpha} \leq (y_N')^{-\alpha}$ for all $i\in\cN$, and \eqref{eqn:y1/yN} has the upper bound
    \begin{equation*}
        \left(\frac{y_1'}{y_N'}\right)^{\alpha+1} \leq \frac{(y_N')^{-\alpha}\sum_{j\in\cN}P_{1j}}{(y_1')^{-\alpha}\sum_{k\in\cN}P_{Nk}} =  \left(\frac{y_1'}{y_N'}\right)^{\alpha},
    \end{equation*}
    which shows $y_1'/y_N' \leq 1$ and contradicts $y_1'>y_N'$. This completes the proof.

    We can additionally prove that that Proposition \ref{prop:uniqueness} holds true for all $\alpha \in (-0.5,0)$ as well. When $\alpha < 0$, we have $(y_N')^{-\alpha} \leq (y_i')^{-\alpha} \leq (y_1')^{-\alpha}$ for all $i\in\cN$, and \eqref{eqn:y1/yN} has the upper bound
    \begin{equation*}
        \left(\frac{y_1'}{y_N'}\right)^{\alpha+1} \leq \frac{(y_1')^{-\alpha}\sum_{j\in\cN}P_{1j}}{(y_N')^{-\alpha}\sum_{k\in\cN}P_{Nk}} = \left(\frac{y_1'}{y_N'}\right)^{-\alpha},
    \end{equation*}
    which leads to $(y_1'/y_N')^{2\alpha +1} \leq 1$. If we additionally have $2\alpha+1 > 0$, or equivalently $\alpha \in (-0.5,0)$, we have $y_1' \leq y_N'$, which contradicts $y_1' > y_N'$, implying $x_i = \mu_i$ for all $i\in\cN$ even when $\alpha \in (-0.5,0)$ and the random walker is \textit{self-attractive}.
\end{proof}

\begin{proof}[\textbf{Proof of Lemma \ref{lemma:strict_lyapunov}}]
    We use the notation $\left[\frac{\vx}{\vmu}\right]^{-\alpha}$ to denote a vector with the $i$'th entry being $\left[ \frac{x_i}{\mu_i} \right]^{-\alpha}$, for any $\alpha\in\R$. For any vector $\vv\in\R^N$, we will occasionally use the notation $[\vv]_i$ to denote its $i$'th entry.
    Taking partial derivative of $w(\vx)$ with respect to $x_i$, we have
    \begin{equation}\label{eqn:partial_V_xi}
        \frac{\partial w(\vx)}{\partial x_i}  = \frac{-2\alpha}{\mu_i}\left[\frac{x_i}{\mu_i}\right]^{-\alpha-1}\left[\mD_{\vmu} \mP\left[\frac{\vx}{\vmu}\right]^{-\alpha}\right]_i
    \end{equation}
    Taking derivative of $w(\vx)$ along trajectories of the ODE system \eqref{eqn:main ode}, we get
    \begin{align}
        \frac{d}{dt}w(\vx)  &= \sum_{i} \frac{\partial w(\vx)}{\partial x_i}\frac{d x_i}{dt} = \nabla w(\vx)^T\vh(\vx) = -2\alpha \sum_{i\in\cN}\left(\frac{1}{\mu_i}\left[\frac{x_i}{\mu_i}\right]^{-\alpha-1}\left[\mD_{\vmu}\mP\left[\frac{\vx}{\vmu}\right]^{-\alpha}\right]_i\right)\cdot(\pi_i(\vx)-x_i) \nonumber \\
        &= -2\alpha \sum_{i\in\cN}\left(\frac{1}{\mu_i}\left[\frac{x_i}{\mu_i}\right]^{-\alpha-1}\left[\mD_{\vmu} \mP\left[\frac{\vx}{\vmu}\right]^{-\alpha}\right]_i\right)\cdot \left( \frac{1}{w(\vx)}\left[\frac{x_i}{\mu_i}\right]^{-\alpha}\left[\mD_{\vmu} \mP\left[\frac{\vx}{\vmu}\right]^{-\alpha}\right]_i  - x_i \right) \nonumber \\
        &= \frac{-2\alpha}{w(\vx)} \sum_{i\in\cN}\left(\frac{1}{\mu_i}\left[\frac{x_i}{\mu_i}\right]^{-\alpha-1}\left[\mD_{\vmu} \mP\left[\frac{\vx}{\vmu}\right]^{-\alpha}\right]_i\right)\cdot \left( \left[\frac{x_i}{\mu_i}\right]^{-\alpha}\left[\mD_{\vmu} \mP\left[\frac{\vx}{\vmu}\right]^{-\alpha}\right]_i  - w(\vx)x_i \right) \nonumber \\
        &= \frac{-2\alpha}{w(\vx)} \sum_{i \in \cN} B_i(\vx) + C_i(\vx), \label{eqn:bi+ci}
    \end{align}
    where
    $$B_i(\vx) = \left(\frac{1}{\mu_i}\left[\frac{x_i}{\mu_i}\right]^{-\alpha-1}\left[\mD_{\vmu} \mP\left[\frac{\vx}{\vmu}\right]^{-\alpha}\right]_i\right) \cdot \left( \left[\frac{x_i}{\mu_i}\right]^{-\alpha}\left[\mD_{\vmu} \mP\left[\frac{\vx}{\vmu}\right]^{-\alpha}\right]_i\right),$$
    and
    $$C_i(\vx) = -\left(\frac{1}{\mu_i}\left[\frac{x_i}{\mu_i}\right]^{-\alpha-1}\left[\mD_{\vmu} \mP\left[\frac{\vx}{\vmu}\right]^{-\alpha}\right]_i\right) \cdot \left(w(\vx)x_i \right).$$
    Define a random variable $Z(\vx)$ which takes values $Z_i(\vx) = \frac{1}{\mu_i}\left[\frac{x_i}{\mu_i}\right]^{-\alpha-1}\left[\mD_{\vmu} \mP\left[\frac{\vx}{\vmu}\right]^{-\alpha}\right]_i$ with probability $x_i$. Then, we can see that
    $$B_i(\vx) = \left(\frac{1}{\mu_i}\left[\frac{x_i}{\mu_i}\right]^{-\alpha-1}\left[\mD_{\vmu} \mP\left[\frac{\vx}{\vmu}\right]^{-\alpha}\right]_i\right) \cdot \left( \frac{1}{\mu_i} \left[\frac{x_i}{\mu_i}\right]^{-\alpha-1}\left[\mD_{\vmu} \mP\left[\frac{\vx}{\vmu}\right]^{-\alpha}\right]_i\right)\cdot x_i
    = Z_i(\vx)^2 x_i,$$
    and as a result,
    \begin{equation}\label{eqn:sum bi}
        \sum_{i \in \cN} B_i(\vx) = \sum_{i \in \cN} Z_i(\vx)^2 x_i = \E[Z(\vx)^2].
    \end{equation}
    We can similarly write $\sum_{i\in\cN} C_i(\vx)$ as
    \begin{align}
        \sum_{i\in\cN} C_i(\vx) &= -w(\vx)\sum_{i \in \cN}\frac{1}{\mu_i} \left[\frac{x_i}{\mu_i}\right]^{-\alpha-1}\left[\mD_{\vmu} \mP\left[\frac{\vx}{\vmu}\right]^{-\alpha}\right]_i\cdot x_i \nonumber\\
        &= -\left(\sum_{k \in \cN}\frac{1}{\mu_k} \left[\frac{x_k}{\mu_k}\right]^{-\alpha-1}\left[\mD_{\vmu} \mP\left[\frac{\vx}{\vmu}\right]^{-\alpha}\right]_k\cdot x_k \right) \left(   \sum_{i \in \cN}\frac{1}{\mu_i} \left[\frac{x_i}{\mu_i}\right]^{-\alpha-1}\left[\mD_{\vmu} \mP\left[\frac{\vx}{\vmu}\right]^{-\alpha}\right]_i\cdot x_i\right) \nonumber\\
        &= -\left(\sum_{k\in\cN} Z_k(\vx) x_k \right) \left(\sum_{i\in\cN} Z_i(\vx) x_i \right) = -\E[Z(\vx)]^2. \label{eqn:sum ci}
    \end{align}
    Substituting \eqref{eqn:sum bi} and \eqref{eqn:sum ci} in \eqref{eqn:bi+ci} gives us
    \begin{equation}\label{eqn:var z}
        \frac{d}{dt}w(\vx) = \frac{-2\alpha}{w(\vx)}\sum_{i\in\cN} B_i(\vx) + C_i(\vx) = \frac{-2\alpha}{w(\vx)}\left(\E[Z(\vx)^2] - \E[Z(\vx)]^2\right) = \frac{-2\alpha}{w(\vx)}\text{Var}[Z(\vx)] \leq 0.
    \end{equation}
    To show that the equality is only achieved at the fixed point, all we need to show is that $\text{Var}[Z(\vx)] = 0 \iff \vx = \vpi(\vx)$. The Variance term is zero if and only if $Z_i = Z_j$ for all $i,j\in\cN$, that is
    \begin{align}
        \frac{1}{\mu_i}\left[\frac{x_i}{\mu_i}\right]^{-\alpha-1}\left[\mD_{\vmu} \mP\left[\frac{\vx}{\vmu}\right]^{-\alpha}\right]_i = \frac{1}{\mu_i}\left[\frac{x_j}{\mu_j}\right]^{-\alpha-1}\left[\mD_{\vmu} \mP\left[\frac{\vx}{\vmu}\right]^{-\alpha}\right]_j 
        &\iff \frac{\pi_i(\vx)}{x_i} = \frac{\pi_j(\vx)}{x_j}  \label{eqn:zi=zj}
    \end{align}
    for all $i,j \in \cN$, where the equality on the left hand size comes by rewriting the form for $Z_i$ as $w(\vx)\pi_i(\vx)/x_i$. Equation \eqref{eqn:zi=zj} is true if and only if $\pi_i(\vx) = x_i$ for all $i\in\cN$, which completes the proof.
\end{proof}

\begin{proof}[\textbf{Proof of Theorem \ref{thm:global_convergence_ode}}]
    The trajectories $\{\vx(t)\}_{t\geq0}$ starting from any $\vx(0)\in\Int(\Sigma)$ are relatively compact and bounded, since the flow of the ODE system \eqref{eqn:main ode} leaves the probability simplex positively invariant. Then, the global asymptotic stability of $\vmu \in \Int(\Sigma)$ follows by application of Proposition \ref{prop:uniqueness}, Lemma \ref{lemma:strict_lyapunov}, and the LaSalle Invariance principle (Theorem 3.1 in \cite{JB_semiflow}, Chapter 5 in \cite{Vidyasagar}, Chapter 3 in \cite{Slotine}).
\end{proof}

\begin{proof}[\textbf{Proof of Lemma \ref{lemma:srrw_jacobian}}]
    Consider $f_i(\vx)$ defined as 
    $$ 
    f_i(\vx) \triangleq \sum \limits_{j\in\cN} \mu_i P_{ij} \left(\frac{x_i}{\mu_i}\right)^{\!-\alpha} \!\left(\frac{x_j}{\mu_j}\right)^{\!-\alpha}
    $$
    for all $i\in\cN$, $\vx \in \Int(\Sigma)$, and let $g(\vx) \triangleq \sum_{k\in\cN} f_k(\vx)$. Then, for $h_i(\vx) = \pi_i(\vx) - x_i$ (the $i$'th entry of $\vh(\vx)$ in \eqref{eqn:main ode}), we have
    \begin{equation}\label{eqn:h_i srrw}
    h_i(\vx) = \frac{f_i(\vx)}{g(\vx)} - x_i,
    \end{equation}
    for all $i\in\cN$ and $\vx\in\Int(\Sigma)$, and its partial derivatives follow
    \begin{equation}\label{eqn:h_i partial}
        \frac{\partial h_i(\vx)}{\partial x_j} = \frac{g(\vx)\frac{\partial f_i(\vx)}{\partial x_j} - f_i(\vx)\frac{\partial g(\vx)}{\partial x_j}}{g(\vx)^2} - \Char_{\{i=j\}}.
    \end{equation}
    To deduce $\frac{\partial h_i(\vx)}{\partial x_j} \big \vert_{\vx = \vmu}$, we evaluate each quantity of the above equation evaluated at $\vx = \vmu$, and then substitute them back in the above expression. We have
    \begin{align*}
        &f_i(\vmu) = \sum_{j\in\cN} \mu_i P_{ij} = \mu_i\sum_{j\in\cN}P_{ij} = \mu_i, &\forall ~i \in \cN, \\
        &g(\vmu) = \sum_{i \in \cN} f_i(\vmu) = \sum_{i\in\cN} \mu_i = 1, & \\
        &\frac{\partial f_i(\vx)}{\partial x_i} \Bigg\vert_{\vx=\vmu} =  -\frac{\alpha}{x_i}\sum_{j\in\cN} \mu_i P_{ij}\left(\frac{x_i}{\mu_i}\right)^{-\alpha}\left( \frac{x_j}{\mu_j}\right)^{-\alpha} - \frac{\alpha}{x_i}\mu_i P_{ii}\left(\frac{x_i}{\mu_i}\right)^{-2\alpha}  \Bigg \vert_{\vx = \vmu}  =  -\alpha  - \alpha P_{ii}, &\forall ~i \in \cN,\\
        &\frac{\partial f_i(\vx)}{\partial x_j} \Bigg\vert_{\vx=\vmu} = -\frac{\alpha}{x_j}\mu_i P_{ij}\left(\frac{x_i}{\mu_i}\right)^{-\alpha}\left( \frac{x_j}{\mu_j}\right)^{-\alpha} \Bigg \vert_{\vx = \vmu} = - \alpha\frac{\mu_i P_{ij}}{\mu_j} = -\alpha P_{ji}, &\forall ~i\neq j \in \cN,\\
        &\frac{\partial g(\vx)}{\partial x_i} \Bigg\vert_{\vx=\vmu} = -\frac{2\alpha}{x_i}\sum_{j\in\cN} \mu_i P_{ij}\left(\frac{x_i}{\mu_i}\right)^{-\alpha}\left( \frac{x_j}{\mu_j}\right)^{-\alpha} \Bigg \vert_{\vx = \vmu} = -2\alpha, &\forall ~i \in \cN.
    \end{align*}
    Substituting the above expressions in \eqref{eqn:h_i partial} and simplifying it yields
    \begin{align*}
        \frac{\partial h_i(\vx)}{\partial x_i} \Bigg\vert_{\vx=\vmu} &= 2\alpha \mu_i - \alpha P_{ii}  -\alpha - 1, &\forall ~i \in \cN, \\
        \frac{\partial h_i(\vx)}{\partial x_j} \Bigg\vert_{\vx=\vmu} &= 2\alpha \mu_i - \alpha P_{ji}  ,      &\forall ~i\neq j \in \cN,
    \end{align*}
    and by rewriting the above in matrix form, we get
    \begin{equation}\label{eqn:srrw jacobian again}
        \mJ(\alpha) = 2 \alpha \vmu \ones^T - \alpha \mP^T - (\alpha + 1)\eye
    \end{equation}
    which is the same as \eqref{eqn:srrw_jacobian}.

    We now prove the eigenvalue result in Lemma \ref{lemma:srrw_jacobian}. For each $\mu_i$, $i \in \{1,\cdots, N\}$, we have
    $$\mJ(\alpha) \vu_i = 2\alpha \vmu \ones^T\vu_i - \alpha {\mP}^T \vu_i - (\alpha + 1)\vu_i.$$
    When $i=N$, that is $\vu_i = \vu_N = \vmu$, then we have $\vmu\ones^T\vu_N = \vu_N\ones^T\vmu = \vmu = \vu_N$, and ${\mP}^T\vu_N = \vu_N$, and so the above equation becomes
    $$\mJ(\alpha) \vu_N = (2\alpha - \alpha - \alpha - 1) \vu_N = (-1)\vu_N.$$
    When $i \neq N$, we have $\vmu\ones^T\vu_i = \vmu\vv_N^T\vu_i = 0$, and ${\mP}^T\vu_i = \lambda_i\vu_i$, and we similarly have
    $$\mJ(\alpha) \vu_i = (0 - \alpha\lambda_{i} - \alpha - 1) \vu_i = (\alpha(-1 - \lambda_i)-1)\vu_i.$$
    Similar steps follow when we start from $\vv_i^T\mJ(\alpha)$ instead, and $\zeta_N = -1$, $\zeta_i = (\alpha(-1 - \lambda_i)-1)$ for all $i\in\{1,\cdots,N-1\}$ are the eigenvalues of $\mJ(\alpha)$ with $\vu_i$ and $\vv_i$ being the corresponding left and right eigenvectors. Since $(-\lambda_i - 1) < 0$ for all $i \in \{1,\cdots,N-1\}$, $\zeta_i$'s follow the same ordering as $\lambda_i$'s and this completes the proof.
\end{proof}


\section{Proof of results in Section \ref{section:stochastic_analysis}}\label{appendix:stochastic analysis}

Before providing the proofs of our main results, we reproduce some key assumptions from \cite{delyon2000stochastic} required to apply Theorems 15 and 25 therein, which are the almost sure convergence and the CLT result respectively.

\begin{itemize}
    \item[(A)] $\vh$ is a continuous vector field on $\cO \subset \R^d$; there exists a non-negative $C^1$ function $w$ such that
    \begin{itemize}
        \item $\nabla w(\vx)^T \vh(\vx) \leq 0$ for all $\vx \in \cO$;
        \item the set $S \triangleq \{ \vx ~|~  \nabla w(\vx)^T \vh(\vx) = 0\}$ is such that $w(S)$ has an empty interior.
    \end{itemize}
    \item[(B)] $\vh$ is a continuous vector field on $\cO \subset \R^d$; there exists a non-negative $C^1$ function $w$ and a compact set $\cK \subset \cO$ such that
    \begin{itemize}
        \item $w(\vx) \to \infty$ if $\vx \to \partial\cO$ or $\|\vx\| \to \infty$;
        \item $\nabla w(\vx)^T \vh(\vx) < 0$ if $x \notin \cK$.
    \end{itemize}
    \item[(C)] The general SA iteration given by $\vx_{n+1} = \vx_n + \gamma_{n+1}[\vh(\vx_n) + \eta_{n+1}]$ is \textit{A-stable} \citep[see][Definition 1]{delyon2000stochastic} and $\vx_n$ converges to some limit $\vx^*$. $\vh$ is $C^1$ in some neighborhood of $\vx^*$ with first derivatives being Lipschitz, and the Jacobian matrix of $\vh$ evaluated at $\vx^*$ has all its eigenvalues with negative real part.
    \item[(D)] The step size is decreasing and satisfies 
    \begin{equation*}
        \begin{cases}
            \text{either} \ \ \ \frac{1}{\gamma_n} - \frac{1}{\gamma_{n-1}} \to 0, \\
            \text{or} \ \ \ \ \ \ \ \ \ \gamma_n n \to 1.
        \end{cases}
    \end{equation*}
    \item[(MS)] (Translated to the non-linear Markov chain setting) For every $\vx \in \cO$, there exists a solution $\mQ(\vx)$ to the \textit{Poisson equation}
    \begin{equation*}
        (\eye - \mK[\vx])\mQ[\vx] = \eye - \ones \vpi(\vx)^T.
    \end{equation*}
    For any compact $\cK \subset \cO$,
    $$\sup_{\vx \in \cK, i \in \cN} \| \mQ[\vx]^T\ve_i \|_2 + \| (\eye - \ones\vx^T)^T\ve_i \|_2 < \infty$$
    and there exists a continuous function $\phi_{\cK}$, $\phi_{\cK}(0) = 0$, such that for any $\vx, \vx' \in \cK$,
    $$\sup_{i \in \cN} \| \left[\mK[\vx]\mQ[\vx]\right]_{\cdot,i} - \left[\mK[\vy]\mQ[\vy]\right]_{\cdot,i} \|_2 \leq \phi_{\cK}(\|\vx - \vy\|_2).$$
\end{itemize}

\begin{proof}[\textbf{Proof of Theorem \ref{thm:first order result}}]
    As mentioned in Section \ref{Section:algorithmic setup}, the SRRW iteration \eqref{eqn:srrw iteration} is a stochastic approximation algorithm with controlled Markovian input, with its step size sequence given by $\gamma_n = \frac{1}{n+1}$. To prove the almost sure convergence, we show that assumptions (A), (B) and (MS) in \cite{delyon2000stochastic} are satisfied, and then, under \ref{assu:boundedness}, apply Theorem 15 therein. 
    As a result of Proposition \ref{prop:uniqueness} and Lemma \ref{lemma:strict_lyapunov}, the set of fixed points, which in our case is the singleton $\{ \mu \} \subset \Int(\Sigma)$, and the Lyapunov function $w:\Int(\Sigma) \to [0,\infty)$ as defined in \eqref{eqn:lyapunov_function} satisfy assumptions (A) and (B) in \cite{delyon2000stochastic}.

    Since $\mK[\vx]$ is irreducible for all $\vx \in \Int(\Sigma)$, the semigroup $\{ e^{t(\mK[\vx] - \eye)} \}_{t \geq 0}$ of the related CTMC kernel $\mK[\vx] - \eye$ converges exponentially towards $\ones\vpi(\vx)^T$ (geometric ergodicity). Thus, for all $\vx \in \Int(\Sigma)$, or equivalently for all $\vx \in \cK$ for any compact $\cK \subset \Int(\Sigma)$, the matrix
    \begin{equation}\label{eqn:poisson_1}
        \mQ[\vx] = \int_{0}^{\infty} \left( e^{t(\mK[\vx] - \eye)} - \ones \vpi(\vx)^T \right)dt
    \end{equation}
    is well defined. Moreover, it solves the Poisson equation; that is
    \begin{align*}
        (\eye - \mK[\vx]) \mQ[\vx]  &= \int_{0}^{\infty} \left( (\eye - \mK[\vx])e^{-t(\eye - \mK[\vx])} - (\eye - \mK[\vx])\ones \vpi(\vx)^T \right)dt \\
                                &= \int_{0}^{\infty} (\eye - \mK[\vx])e^{-t(\eye - \mK[\vx])}dt =  \left(\eye-e^{t(\mK[\vx] - \eye)}\right)\Big|_0^\infty = \eye - \ones \vpi(\vx)^T
    \end{align*}
    where the second equality is because $\mK[\vx]\ones = \ones$, and the last inequality is because $-e^{t(\mK[\vx] - \eye)}$ the semi-group operator of an ergodic CTMC, which goes to $\ones \vpi(\vx)^T$ as $t \to \infty$. The solution of the Poisson equation $\mQ[\vx]$, as well as the state dependent update (matrix) $\eye - \ones \vx^T$, have bounded entries for all $\vx\in\cK \subset \Int(\Sigma)$, which implies that
    $$\sup_{\vx \in \cK, i \in \cN} \| \mQ[\vx]^T\ve_i \|_2 + \| (\eye - \ones\vx^T)^T\ve_i \|_2 = \| \mQ[\vx]_{\cdot,i} \|_2 + \| \bfdelta_i - \vx \|_2 < \infty.$$
    Moreover, since $\mK[\vx]$ and $\mQ[\vx]$ are continuous at every $\vx\in\Int(\Sigma)$, they are also Lipschitz in any compact $\cK\subset \Int(\Sigma)$. Thus for each $\cK$, there exists a constant $C_{\cK}$ such that for any $\vx,\vy \in \cK$,
    $$\sup_{i \in \cN} \| \left[\mK[\vx]\mQ[\vx]\right]_{\cdot,i} - \left[\mK[\vy]\mQ[\vy]\right]_{\cdot,i} \|_2 \leq C_{\cK}\|\vx - \vy\|_2.$$
    With this, we satisfy the assumption (MS), and the result follows by application of Theorem 15 in \cite{delyon2000stochastic}.
\end{proof}

\begin{proof}[\textbf{Proof of Theorem \ref{thm:second order result}}]
    We apply Theorem 25 in \cite{delyon2000stochastic} to prove \eqref{eqn:CLT of iterates}, and then analyzing the form of the resulting co-variance matrix to prove \eqref{eqn:covariance matrix}. Applying Theorem 25 in \cite{delyon2000stochastic} firstly involves checking that assumptions (C), (D) and (MS) therein are satisfied. Since we already showed that (MS) is satisfied while proving Theorem \ref{thm:first order result}, we focus on proving (C) and (D).

    Our choice of step size, $\gamma_n = \frac{1}{n+1}$ satisfies $\sum_{n\in\N_0} \gamma_n = \infty$ and $\sum_{n \in \N_0} \gamma_n^2 < \infty$. Besides,
    $$\gamma_n - \gamma_{n+1} = \frac{1}{n+1} - \frac{1}{n+2} = \frac{(n+2) - (n+1)}{(n+1)(n+2)} \leq \frac{1}{(n+1)^2}.$$
    Then, we have $\sum_{n \in \N_0} | \gamma_n - \gamma_{n+1} | \leq \sum_{n \in \N_0} {1}/{(n+1)^2} < \infty$, and therefore (D) in \cite{delyon2000stochastic} is satisfied.

    Since our mean field $F(\vx)$ is differentiable everywhere in $\Int(\Sigma)$ and therefore continuous, it is Lipschitz for all compact subsets $\cK \subset \Int(\Sigma)$, and thus also Lipschitz over some neighborhood of $\vmu \in \Int(\Sigma)$. Moreover linear stability of  $\vmu$ follows from the global stability of shown in Theorem \ref{thm:global_convergence_ode}, and all eigenvalues of $\mJ(\alpha)$ have negative real parts as shown in Lemma \ref{lemma:srrw_jacobian}. This ensures that (C) in \cite{delyon2000stochastic} is satisfied.

    In order to obtain \eqref{eqn:CLT of iterates} via Theorem 25 in \cite{delyon2000stochastic}, it remains to show that there exists a bounded $\mW \in \R^{N\times N}$ which solves 
    \begin{equation*}
        [\eye - \mK[\vmu]]\mW = [\eye - \ones \vmu^T]\mS,
    \end{equation*}
    where $\mS = \mQ[\vmu]\mQ[\vmu]^T - \mK[\vmu]\mQ[\vmu]\mQ[\vmu]^T\mK[\vmu]^T$, and $\mQ[\cdot]$ is as defined in \eqref{eqn:poisson_1}. By setting $\mW = \mQ[\vmu]\mS$, we can check that $[\eye - \mK[\vmu]]\mW = [\eye - \mK[\vmu]]\mQ[\vmu]\mS = [\eye - \ones \vpi(\vmu)^T]\mS = [\eye - \ones \vmu^T]\mS$, where the second equality comes from the first equation (poisson equation) in the (MS) condition by setting $\vx = \vmu$, and the last equality comes from the $\vmu$ being the unique solution to the fixed point equation $\vpi(\vx)=\vx$. The boundedness of $\mW$ is ensured by that of $\mQ[\vmu]$ and $\mK[\vmu]$. This completes the proof of \eqref{eqn:CLT of iterates}, and we will now show \eqref{eqn:covariance matrix}.

    From Theorem 25 in \cite{delyon2000stochastic}, the matrix $\mV(\alpha)$ solves the Lyapunov equation $\mU + (\mJ(\alpha) + \eye/2)\mV(\alpha)+\mV(\alpha)(\mJ(\alpha) + \eye/2)^T$, and is therefore given by 
    \begin{equation}\label{eqn:lyapunov_sol}
        \mV(\alpha) = \int_0^\infty e^{t(\mJ(\alpha)+\eye/2)} \mU e^{t(\mJ(\alpha)+\eye/2)^T}dt.
    \end{equation} From Theorem 25 in \cite{delyon2000stochastic}, the matrix $\mU$ is given by $\mU = \ones\vmu^T\mS$ and is called the asymptotic (sampling) co-variance for function $\vf:\cN\to\R^d$, with $\vf(i) = \bfdelta_i - \vmu$ for all $i\in \cN$.\footnote{The map $\vf$ can be thought of as a function to be sampled at $X_{n+1}$ and corresponds to the update rule evaluated at $\vx_n = \vmu$, that is, $\mF_{X_{n+1},\cdot} = \vf(X_{n+1}) = \bfdelta_{X_{n+1}} - \vmu$.} To obtain the closed form of $\mU$ in terms of eigenvalues and eigenvectors of the probability matrix $\mP$, we first provide  Lemma 6.3.7 in \cite{bremaudmarkov2020}), but re-written for vector-valued functions instead.
    \begin{lemma}[Lemma 6.3.7 in \cite{bremaudmarkov2020}]
        Let $\{X_k\}_{k\geq 0}$ be an ergodic Markov chain (reversible) with finite state space $[n]$, transition probability matrix $\mP$ and stationary distribution $\vmu$. For any function $\vf:[n]\to\R^d$, we have
        \begin{equation}
            \mU(\vf) = 2\mF^T \text{diag}(\vmu) \mZ \mF - \mF^T \text{diag}(\vmu) \mF - \mF^T \vmu \vmu^T \mF,
        \end{equation}
        where $\mU(\vf)$ is the asymptotic (sampling) covariance matrix for function $\vf$, matrix $\mF$ is given by $\mF \triangleq [\vf(1), \cdots, \vf(n)]^T$, and
        $\mZ \triangleq (\eye - \mP + \ones \vmu^T)^{-1}$.
        Moreover, since $\mP$ is reversible, we have (equation 6.34 in \cite{bremaudmarkov2020})
        \begin{equation}
        \mU(\vf) = \sum \limits_{k=1}^{n-1} \frac{1+\lambda_k}{1-\lambda_k} \mF^T\vu_k \vu_k^T\mF,
        \end{equation}
        where $\vu_k$ are the left eigenvectors of $\mP$ with $\vu_n = \vmu$.
    \end{lemma}

    For our update rule \eqref{eqn:main ode}, we have $\mF = \eye - \ones\vmu^T$. The asymptotic (sampling) covariance matrix $\mU \triangleq \mU(\vf)$ is the same for all $\alpha \geq 0$, since it is only dependent on $\mF$ and the transition kernel $\mP = \mK[\vmu]$. Thus, we can write down $\mU$ as

    \begin{equation}\label{eqn:spectral form U}
        \mU = \sum \limits_{k=1}^{N-1} \frac{1+\lambda_k}{1-\lambda_k} (\eye - \vmu \ones^T)\vu_k \vu_k^T(\eye -  \ones\vmu^T) = \sum \limits_{k=1}^{N-1}\frac{1+\lambda_k}{1-\lambda_k}\vu_k \vu_k^T
    \end{equation}
    where the last equality is because $\vu_N=\vmu$ and $\vv_N=\ones$, and since $\vu_i^T\vv_j = 0$ for all $i\neq j$.

    Lemma \ref{lemma:srrw_jacobian} allows us to write down the spectral decomposition of $e^{t(\mJ(\alpha) + \eye/2)}$ as
    \begin{equation}\label{eqn:decomposition of exp 2}
        e^{t(\mJ(\alpha)+\eye/2)} = \sum_{i \in \cN} e^{t (\zeta_i+1/2)}\vu_i\vv_i^T = e^{-1/2}\vu_N\vv_N^T + \sum_{i=1}^{N-1} e^{t(\alpha(-1-\lambda_i) -1/2)}\vu_i\vv_i^T.
    \end{equation}

    Substituting \eqref{eqn:spectral form U} and \eqref{eqn:decomposition of exp 2} in \eqref{eqn:lyapunov_sol}, we get
    \begin{align*}
        \mV(\alpha) &=\!\! \int_0^\infty \!\! \left(\!e^{-\frac{1}{2}}\vu_N\vv_N^T \!+\!\! \sum_{i=1}^{N-1} e^{t(\alpha(-1-\lambda_i) -\frac{1}{2})}\vu_i\vv_i^T \!\right)\!\!\left(\!\sum_{k=1}^{N-1}\frac{1+\lambda_k}{1-\lambda_k}\vu_k \vu_k^T\!\right) \! \!\left(\!e^{-\frac{1}{2}}\vu_N\vv_N^T \!+ \!\! \sum_{i=1}^{N-1} e^{t(\alpha(-1-\lambda_i) -\frac{1}{2})}\vu_i\vv_i^T \! \right)^{\!\!\!\!T}\! dt \\
        &=\!\!  \int_0^\infty \left(\sum_{i=1}^{N-1} e^{t(\alpha(-1-\lambda_i) -1/2)}\vu_i\vv_i^T \right)\left(\sum_{k=1}^{N-1}\frac{1+\lambda_k}{1-\lambda_k}\vu_k \vu_k^T\right) \left(\sum_{i=1}^{N-1} e^{t(\alpha(-1-\lambda_i) -1/2)}\vu_i\vv_i^T \right)^T dt \\
        &=\!\!  \int_0^\infty \sum_{i=1}^{N-1} e^{2t(\alpha(-1-\lambda_i) -1/2)}\frac{1+\lambda_i}{1-\lambda_i}\vu_i\vu_i^T dt 
        = \sum_{i=1}^{N-1} \frac{1}{2\alpha (1 + \lambda_i) + 1} \cdot \frac{1+\lambda_i}{1-\lambda_i}\vu_i\vu_i^{T}
    \end{align*}
    where the first three equalities follow from orthonormality of the left and right eigenvectors. The last equality comes from swapping the summation with the integral,\footnote{From Fubini's theorem, the order of summation and integrals can be swapped if the summands are all positive terms - which stands true for our case.} and evaluating the latter. This completes the proof.
\end{proof}

\begin{proof}[\textbf{Proof of Corollary \ref{cor:avr ordering}}]
    For any $\alpha > 0$ and any vector $\vx \in \R^N$, we have
    $$\vx^T\mV(\alpha)\vx = \sum_{i=1}^{N-1}\frac{1}{2\alpha (1 + \lambda_i) + 1} \cdot \frac{1+\lambda_i}{1-\lambda_i} \vx^T\vu_i\vu_i^T\vx  < \frac{1+\lambda_i}{1-\lambda_i} \vx^T\vu_i\vu_i^T\vx = \vx^T\mV(0)\vx,$$
    where the inequality is because $\lambda_i \in (-1,1)$, and as a result, $\alpha(-\lambda_i - 1) < 0$ for all $i \in \{1,\cdots,N-1\}$. In fact, the ordering is monotone in $\alpha>0$. This completes the proof.
\end{proof}

\begin{proof}[\textbf{Proof of Corollary \ref{corollary:lp_norm}}]
    Note that every $L^p$ norm $\| \cdot \|_p$ is a continuous and bounded function on $\Sigma$. The implication then follows by a direct application of the continuous mapping theorem \citep[see][Theorem 2.3]{van2000asymptotic}  to the CLT shown in Theorem \eqref{thm:second order result}.
\end{proof}

\begin{proof}[\textbf{Proof of Corollary \ref{cor:estimator and mse clt}}]
    Recall that in Section \ref{Section:algorithmic setup}, we had redefined $\vx_n$ as $\vx_n \triangleq \frac{1}{n+1}(\vx_0 + \sum_{k=1}^n \bfdelta_{X_k})$, where $\vx_0 = \vnu \in \Int(\Sigma)$ was set a-priori. We redefine the actual empirical distribution of the process as $\hat \vx_n \triangleq \frac{1}{n}\sum_{k=1}^n \bfdelta_{X_k}$, and rewrite $\vx_n$ as
    \begin{equation}
        \vx_n = \frac{1}{n+1}\vx_0 + \frac{n}{n+1}\hat \vx_n
    \end{equation}

    It is enough to show the almost sure convergence and CLT result for $\hat \vx_n$, since the result for $\psi_n(g) = \vg^T\hat\vx_n$ follows from the boundedness assumption for $\vg$ (which ensures square summability). From Theorem \ref{thm:first order result}, and because $\vx_0/n+1 \to \0$, as $n\to \infty$, we have that $\hat{\vx}_n n/(n+1) \to \vmu$ almost surely. Multiplying $\hat{\vx}_n \cdot n/(n+1)$ by $(n+1)/n$,  and since $(n+1)/n \to 1 < \infty$, we obtain that $\hat{\vx}_n \to \vmu$ almost surely.

    From the CLT in Theorem \ref{thm:second order result}, we have $\sqrt{n}(\vx_n - \vmu) \to \cN(0,\mV(\alpha))$ in distribution. We will again break down $\vx_n$ to obtain 
    $$\frac{n\sqrt{n}}{n+1}\hat \vx_n - \sqrt{n}\vmu \xrightarrow[n\to\infty]{dist} \cN(0,\mV(\alpha)),$$ 
    where the $\vx_0$ term is not present since $ \frac{\sqrt{n}}{n+1}\vx_0 \to \0$. We multiply the above by $(n+1)/n$ to get
    $$\frac{n+1}{n} \left( \frac{n\sqrt{n}}{n+1}\hat \vx_n - \sqrt{n}\vmu\right) \xrightarrow[n\to\infty]{dist} \cN(0,\mV(\alpha)),$$
    where the convergence holds because $n+1/n \to 1<\infty$ and by applying Slutsky's theorem. Additionally, observe that
    $$\frac{n+1}{n} \left( \frac{n\sqrt{n}}{n+1}\hat \vx_n - \sqrt{n}\vmu\right) = \sqrt{n}(\hat{\vx}_n - \vmu) - \frac{1}{n}\vmu,$$
    where the term $\frac{1}{n}\vmu \to \0$, implying that
    $\sqrt{n}(\hat{\vx}_n - \vmu) \xrightarrow[n\to\infty]{dist} \cN(0,\mV(\alpha)).$
    Now we left multiply $\hat{\vx}_n$ by vector $\vg^T$ and use continuous mapping theorem such that
    $$\sqrt{n}(\psi_n(g) - \vg^T\vmu) \xrightarrow[n\to\infty]{dist} \cN(0,\vg^T\mV(\alpha)\vg).$$
    This completes the proof.
\end{proof}

\begin{proof}[\textbf{Proof of Corollary \ref{cor:avr bound}}]
    For $\mV(\alpha)$ as given by \eqref{eqn:covariance matrix} for any $\alpha \geq 0$, we have
    \begin{align*}
        \vg^T\mV(\alpha)\vg &= \sum_{i=1}^{N-1} \frac{1}{2\alpha (1 + \lambda_i) + 1} \cdot \frac{1+\lambda_i}{1-\lambda_i}(\vg^T\vu_i)^2 
        = \E[f(\Lambda)g(\Lambda)]\sum_{i\in\cN}(\vg^T\vu_i)^2,
    \end{align*}
    where $\Lambda$ is the random variable as defined in the corollary statement, $f(x)\triangleq\frac{1}{2\alpha(1+x)+1}$ and $g(x)\triangleq\frac{1+x}{1-x}$.
    Similarly, for $\alpha = 0$, we have
    $$\vg^T\mV(0)\vg = \sum_{i=1}^{N-1} \frac{1+\lambda_i}{1-\lambda_i}(\vg^T\vu_i)^2 = \E[g(\Lambda)] \sum_{i\in\cN}(\vg^T\vu_i)^2.$$
    Dividing the two equations gives us
    $$\frac{\vg^T\mV(\alpha)\vg}{\vg^T\mV(0)\vg} = \frac{\E[f(\Lambda)g(\Lambda)]}{\E[g(\Lambda)]} \leq \frac{\E[f(\Lambda)]\E[g(\Lambda)]}{\E[g(\Lambda)]} = \E[f(\Lambda)] = E\left[\frac{1}{2\alpha(1+\Lambda)+1}\right],$$
    where the inequality comes by application of the FKG inequality \citep[see][]{fortuin1971correlation} to the numerator, where $f(\Lambda)$ and $g(\Lambda)$ are positive functions of $\Lambda$, and decreasing and increasing in $\Lambda$ respectively. This completes the proof.
\end{proof}




\section{Scale invariance of SRRW transition probabilities with polynomial form of $r_{\mu_i}(x_i)$} \label{appendix:discussion on polynomial}

In Section \ref{section:introduction}, we briefly mentioned that the polynomial form of $r_{\mu_i}(x_i)$ as introduced in \eqref{eqn:polynomial_repellence} for all $i\in\cN$ is the only possible choice where transition probabilites $\mK_{ij}[\vx]$ are \textit{scale invariant} - the entries of the target distribution $\mu_i$ and empirical measure $x_i$ only need to be known up to a constant multiple for neighboring nodes of the random walker's position at each time step. This is equivalent to saying that for any $\mu_i \in \R$ (for any $x_i \in \R$), we have $r_{\mu_i}(Cx_i) = g(C)r_{\mu_i}(x_i)$ (we have $r_{C\mu_i}(x_i) = g(C)r_{C\mu_i}(x_i)$) for some function $g:\R\to\R$. Indeed, observe that for any $i,j\in\cN$ and $r_{\mu_i}(x_i)$ as discussed, we have
$$K_{ij}[C\vx] = \frac{P_{ij}r_{\mu_j}(Cx_j)}{\sum_k P_{ik}r_{\mu_k}(Cx_k)} = \frac{P_{ij}g(C)r_{\mu_j}(x_j)}{\sum_k P_{ik}g(C)r_{\mu_k}(x_k)} = \frac{P_{ij}r_{\mu_j}(x_j)}{\sum_k P_{ik}r_{\mu_k}(x_k)} = K_{ij}[\vx].$$
The following result shows such scale-invariance is only a property of polynomial choice of $r_{\mu_i}(x_i)$ for all $i\in\cN$.
\begin{proposition}\label{prop:scale_invariance}
    For any function $r:\R \to \R$, there exists a function $g:\R \to \R$ such that 
    \begin{equation}\label{eqn:scale_invariance}
        r(C_1x) = g(C_1)r(x)
    \end{equation}
    for any $C_1 \in \R$ if and only if $r(x)$ is of the form $r(x) = C_2x^\alpha$ for any $C_2, \alpha \in \R$.
\end{proposition}
\begin{proof}
    The reverse implication of the statement is true for any polynomial function $r:\R \to \R$ of type $r(x) = C_2x^\alpha$ with $g(C_1) = C_1^\alpha$. We now prove the forward direction. Differentiating \eqref{eqn:scale_invariance} on both sides, we get
    \begin{equation}\label{eqn:scale_invariance_derivative}
        C_1 r'(C_1 x) = g(C_1) r'(x)
    \end{equation}
    and dividing the two sides of \eqref{eqn:scale_invariance} by those of \eqref{eqn:scale_invariance_derivative} yields
    \begin{equation}\label{eqn:scale_invariance_divide}
        \frac{r(C_1 x)}{r'(C_1 x)} = C_1\frac{r(x)}{r'(x)}.
    \end{equation}
    Setting $f(x) \triangleq r(x)/r'(x)$ for all $x \in \R$, substituting in \eqref{eqn:scale_invariance_divide}, and then differentiating once more gives us
    \begin{equation}
        f'(C_1 x) = f'(x)
    \end{equation}
    for all $C_1,x\in\R$. This is only possible if $f:\R \to \R$ is a linear function, which by its definition is only possible if $r(x)$ is a polynomial function of the type $r(x) = C_2x^\alpha$, for some $C_2,\alpha \in \R$.
\end{proof}


\section{Discussion on Assumption \ref{assu:boundedness}} \label{appendix:discussion on boundedness}

We first describe the modified stochastic approximation procedure with restarts of the process upon hitting the boundaries of a sequence of (expanding) compact subsets of $\Int(\Sigma)$. Define a sequence of compact subsets $\{\cK_n\}_{n\in\N_0}$ of $\Int(\Sigma)$ such that $\cK_n \subset \cK_{n+1}$ for all $n\in\N_0$, and $\cup_{n\in\N_0} \cK_{n} = \Int(\Sigma)$. Let $\{ \bar \bfgamma^m \}_{m\in\N_0}$ be a family of step size sequences, where $\bar \bfgamma^m \triangleq \{\gamma_{k,m}\}_{k \in \N_0}$ for all $m\in\N_0$, with $\gamma_{k,m} \triangleq \gamma_{k+m,0} \triangleq 1/(k+m+2)$, for all $k,m\in \N_0$.

Setting $\vx_0 \in \cK_0$, where $\cK_0$ is the \textit{active set}, and setting the step-size sequence to be $\bar \bfgamma^0$, we run the iteration
\begin{equation}\label{eqn:SA_iteration_truncation_1}
    \vx_{n+1} = (1-\gamma_{n,m})\vx_n + \gamma_{n,m}\bfdelta_{X_{n+1}},
\end{equation}
where $m=0$ and $X_{n+1} \sim K_{X_n,\cdot}[\vx_n]$, until the iterate leaves the active set $\cK_0$. Upon this event (also called a truncation), we `expand' the active set by incrementing its index and setting it to be $\cK_1$, set the new step size sequence to be $\bfgamma^1$, and restart the iteration \eqref{eqn:SA_iteration_truncation_1} with these incremented active sets and step size sequences, and with a new initial point $\vx_0\in\cK_0$, in an \textit{i.i.d.} manner upon each restart. This process of truncation and restarts keeps repeating, and as part of the proof of our first order convergence results, we show that the number of restarts is always finite $\prob_{\vx_0,X_0}$ - almost surely.

This SA procedure with truncations can also be written more comprehensively. Let the step-size sequence be $\bar \bfgamma \triangleq \{ \gamma_k \}_{k \in \N_0}$, where $\gamma_k \triangleq \gamma_{k,0} = 1/(k+2)$. At each step $n\in\N_0$, let $\varsigma_n,\kappa_n$ and $\nu_n$ be \textit{counters} keeping track of the step-size index, the index of the active set, and the number of iterations since the last truncation event, respectively. With $\vx_0 \in \cK_0$, $X_0 \in \cN$ as before, and $\varsigma_0=0$, $\kappa_0=0$, and $\nu_0 = 0$, the SA procedure with truncations can be written as
\begin{equation}\label{eqn:SA_iteration_trucnation_2}
    \begin{split}
        &\text{set:} ~~~~~~~~~~~\vx_{n+\frac{1}{2}} = \vx_n + \gamma_{\varsigma_n + 1}(\bfdelta_{X_{n+1}} - \vx_n), \\
        &\text{update:} ~~~~~(\vx_{n+1},\varsigma_{n+1},\kappa_{n+1},\nu_{n+1}) = 
        \begin{cases}
            (\vx_{n+\frac{1}{2}},\varsigma_{n}\!+\!1,\kappa_{n},\nu_{n}\!+\!1),~~\text{if}~ \vx_{n+\frac{1}{2}} \in \cK_{\kappa_n} \\
            (\vx_{0},\varsigma_{n}\!+\!1-\!\nu_n,\kappa_{n}\!+\!1,0),~~\text{otherwise}.
        \end{cases}
    \end{split}
\end{equation}

\begin{proposition}\label{prop:truncation}
    The iterates $\{\vx_n\}_{n \geq 0}$ of \eqref{eqn:SA_iteration_trucnation_2} satisfy assumption \ref{assu:boundedness}.
\end{proposition}

Before stating the proof, we make the case for why our SRRW iterates in \eqref{eqn:srrw iteration} are almost surely contained within compact subsets\footnote{The compact subset may depend on the sample path.} of $\Int(\Sigma)$, without the need of any truncation at boundaries of an increasing sequence of compact sets. Let the increasing sequence of compact sets be given by $\cK_n \triangleq \left\{ \vx \in \Int(\Sigma) | x_i \in \left[ \frac{1}{n+M}, 1-\frac{1}{n+M} \right] \right\}$, where $M$ can be any positive real number. As explained in \cite{Andrieu2005Stability}, the requirement that $\vx_0\in\cK_0$ is under the condition that $\cK_0$ is a subset of a region where iterates eventually experience a \textit{positive drift} towards the equilibrium point. However, the uniqueness of the $\vmu \in \Int(\Sigma)$ as the fixed point of our mean-field ODE shown in Lemma \ref{prop:uniqueness}, and the form of our \textit{strict} Lyapunov function as in Lemma \ref{lemma:strict_lyapunov} allows us to get rid of this requirement, allowing $\cK_0$ to be \textit{any} compact subset of $\Int(\Sigma)$. This allows us to choose the parameter $M$ defining the size of $\cK_0$ to be large enough so that $\cK_0 \approx \Int(\Sigma)$ and as a result, $\cK_n \approx \Int(\Sigma)$ for all $n\geq 0$. In this manner, the effect of increasing the truncation set is made nearly redundant, and are likely not the key factor to maintaining stability of iterates $\{\vx_n\}_{n\geq 0}$ of \eqref{eqn:SA_iteration_trucnation_2}.

\begin{proof}[\textbf{Proof of Proposition \ref{prop:truncation}}]
    We prove this result for any step size sequence $\{\gamma_n\}_{n\geq 0}$ satisfying
    \begin{itemize}
        \item[\textbf{B1}] $\sum \limits_{k \geq 0} \gamma_k = \infty$, and $\sum \limits_{k \geq 0} \gamma_k^{2-\epsilon} < \infty$ for some $\epsilon \in (0,1)$.
    \end{itemize}
    Which includes the step size $\gamma_n = \frac{1}{n+1}$ considered in our paper. The above assumption is only slightly stricter than the typical one where $\epsilon = 0$, such as (A4) in \cite{fort2015central} and (D) in \cite{delyon2000stochastic} , and we show that $\epsilon$ need only be very small. Therefore in practice, B1 is nearly indistinguishable from (A4) in \cite{fort2015central} and (D) in \cite{delyon2000stochastic}. For our choice of step-size $\gamma_n = \frac{1}{n+1}$, there exists $\epsilon>0$ small enough such that $(\gamma_n, \epsilon)$ satisfy B1.

    We first introduce a sequence $\{\varepsilon_n\}$ where $\varepsilon_n = 2\gamma_n^\delta$ for some $\delta \in (0,1)$ (and thus $2\gamma_{n} < \varepsilon_{n}$). The condition for acceptance of $\vx_{n+\frac{1}{2}}$ can then be rewritten as requiring $\|\vx_{n+\frac{1}{2}} - \vx_n\| < \epsilon_{\varsigma_n}$ along with $\vx_{n+\frac{1}{2}} \in \cK_{\kappa_n}$, where the former is trivially satisfied since $\|\vx_{n+\frac{1}{2}} - \vx_n\| \leq \gamma_{\varsigma_n}\|\delta_{X_{n+1} - \vx_n} \| \leq 2\gamma_{\varsigma_n} < \varepsilon_{\varsigma_n}$. With this modification, update rule \eqref{eqn:SA_iteration_trucnation_2} is then a special case of the general SA algorithm described in Section 3.2 in \cite{Andrieu2005Stability}. The rest of the proof will then be checking that the assumptions required for applying Theorem 5.4 in \cite{Andrieu2005Stability} are satisfied.

    Assumption (A1) in \cite{Andrieu2005Stability} is satisfied with $V(\vx)$ in \eqref{eqn:lyapunov_function} as the choice of Lyapunov function. A1(i)--(iv) in \cite{Andrieu2005Stability} all follow from Lemma \ref{lemma:strict_lyapunov}, coupled with the fact that $\vx^*=\vmu \in \Int(\Sigma)$ is the unique fixed point; the set of equilibria $\cL = \{\vmu\}$ is a singleton and therefore a closed set with non-empty interior, and the constants $M_0$ and $M_1$ can be any real numbers such that $V(\vd) < M_0 < M_1 < \infty$.

    Assumption (A2) in \cite{Andrieu2005Stability} is naturally satisfied by the construction of our SA algorithm, since $\mK[\vx]$ is irreducible for any $\vx \in \Int(\Sigma)$.

    We now check the set of assumptions (DRI) in \cite{Andrieu2005Stability}. The condition (DRI1) is satisfied by any ergodic Markov chain, and therefore also by $\mK[\vx]$ for any $\vx\in \cK \subset \Int(\Sigma)$ with $V(i) = 1$ for all $i \in \cN$, where $\cK$ is any compact subset of $\Int(\Sigma)$. Condition (DRI2) when translated to our setting requires checking for any compact $\cK,\cK' \subset \Int(\Sigma)$ that
    $$\sup_{\vx \in \cK} \|\bfdelta_{i} - \vx\| \leq C_1, ~~\text{and}~~ \sup_{(\vx,\vy) \in \cK \times \cK'} \|\vy - \vx\|^{-\beta}\|\vy - \vx\| \leq C_1$$
    for some $C\in\R$ and $\beta \in [0,1]$. This clearly holds true with $C_1=2$ and $\beta = 1$. The condition (DRI3) when translated to our setting requires showing for any $(\vx,\vy) \in \cK \times \cK'$ that there exists $C_2 \in \R$ such that
    $$\| \mK[\vx]\vu - \mK[\vy]\vu\| \leq C_2 \|\vu\| \|\vx - \vy\|^\beta, ~~~~\forall \vu \in \R^N, \|\vu\|< \infty.$$
    This is again clearly holds with $\beta=1$ and for some $C_2<\infty$, since $\mK[\cdot]$ is not just continuous but also Lipschitz in any compact subset of $\Int(\Sigma)$. Thus, update rule \eqref{eqn:SA_iteration_trucnation_2} satisfies (DRI) which implies (A3) in \cite{Andrieu2005Stability}.

    In order to satisfy (A4) in \cite{Andrieu2005Stability}, we need to show that the sequences $\{\gamma_n\}_{n \geq 0}$ and $\{\varepsilon_n\}_{n \geq 0}$ are non-increasing,  positive, and satisfy $\sum_{k \geq 0} \gamma_k = \infty$, $\lim_{k \to \infty} \varepsilon_k = 0$, and
    $$\sum_{k \geq 0} \left\{ \gamma_k^2 + \gamma_k \epsilon^a + (\epsilon_k^{-1}\gamma_k)^p \right\} < \infty.$$
    Here, we can set $0<a<\beta=1$ and $p \geq 2$ (these can be deduced from the drift conditions (DRI), as discussed in Section 6 of \cite{Andrieu2005Stability}). By setting $\varepsilon_n = 2\gamma_n^{\delta}$ for some small $\delta \in (0,1)$, the condition boils down to choosing $\alpha \in (0,1)$ such that
    $$\sum_{k \geq 0} \gamma_k^{a\delta + 1} < \infty.\footnote{This implies summability of $\gamma_k^2$ term, while that of the $(\varepsilon_k^{-1}\gamma_k)^p$ is ensured by setting $p\geq2$ to a large enough value.}$$
    By setting $\epsilon = 1-a\delta > 0$ in assumption B1, we can see that (A4) in \cite{Andrieu2005Stability} is satisfied. Note that $a \delta$ can be chosen to be very close to $1$, which implies that in practice, B1 is no stricter than the square summability assumption typically found in SA literature.

    From Theorem 5.4 in \cite{Andrieu2005Stability}, we have that the number of restarts are finite $\prob_{\vx_0,X_0}$-almost surely, implying that the iterate sequence will eventually remain in a compact subset of $\Int(\Sigma)$. 
\end{proof}


\section{Additional numerical results}\label{appendix:additional_simulation_results}

In this appendix, we provide supplementary numerical results to those in Section \ref{section:numerical_results}. We compare the sampling performance SRRW with the MHRW as its base chain, with the Metropolis-Hastings with delayed acceptance (MHDA) sampler introduced in \cite{LeeSIGMETRICS12}. The MHDA works to reduce the inherent backtracking probability of MHRW by interacting with its most recently visited state \citep[see][Section 4.3 for the detailed definition of the transition probabilities]{LeeSIGMETRICS12}. We provide our numerical results for two different graphs over the same set of nodes. As mentioned in the captions for Figures \ref{fig:2a} and \ref{fig:2b}, the base MHRW chain for both these graphs have different mixing properties.

The simulation setup is similar to the one in Section \ref{section:numerical_results}, the numerical results are shown Figure \ref{fig:experiments2} where we focus on the MSE of the estimators. For values of $\alpha > 0$ which are only moderately large, the SRRW significantly outperforms the MHDA sampler, showing the performance improvement in the asymptotic regime from interacting with the entire history of the random walker instead of only the most recently visited state. We also observe that performance benefits of SRRW with larger $\alpha$ kick in earlier when the the underlying base chain is faster mixing, which is the case on the Erdos-Renyi random graph in Figure \ref{fig:2b}.

\begin{figure*}[!b]
    \centering 
    \subfigure[Convergence of $\psi_n(g)$ to the ground truth $\vg^T\ones/N$  
    for the wikiVote graph ($889$ nodes, $2914$ edges, SLEM $=0.99$). Base MHRW mixes slower.]{\includegraphics[width=0.55\textwidth]{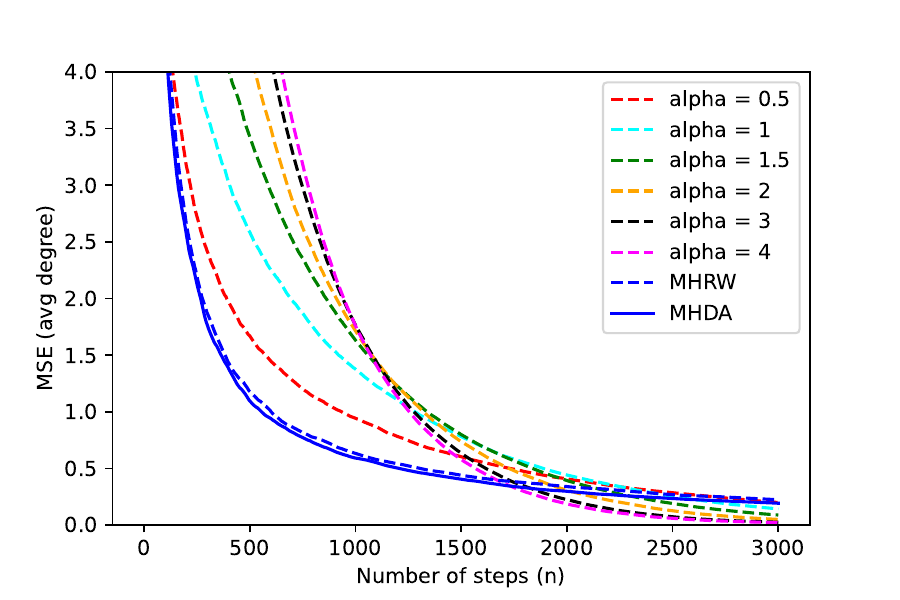} \label{fig:2a}}
    \subfigure[Convergence of $\psi_n(g)$ to the ground truth $\vg^T\ones/N$ 
    for an Erdos-Renyi random graph ($889$ nodes, $3927$ edges, SLEM $=0.93$). Base MHRW mixes faster.]{\includegraphics[width=0.55\textwidth]{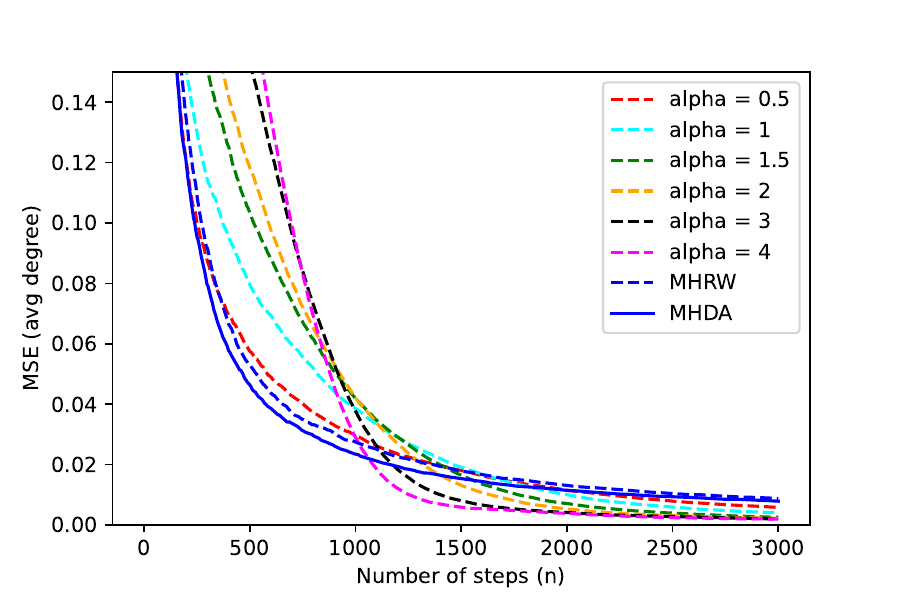} \label{fig:2b}}
    \caption{Simulations of the SRRW process for values of $\alpha \in [0,4]$, where $\alpha = 0$ corresponds to \textit{MHRW} - the underlying Metropolis-Hastings \textit{base} chain, with no self-repellence properties.}
    \label{fig:experiments2}
\end{figure*}


\end{document}